\documentclass{amsart}
\usepackage{amssymb,amstext,amsmath,amscd,amsthm,amsfonts,enumerate,graphicx,latexsym,stmaryrd,multicol}
    
\usepackage[usenames]{color}
\usepackage[all]{xy}
\usepackage[colorlinks]{hyperref}
\usepackage{tikz-cd}  
\usepackage{tikz} 
\usepackage{hyperref}

\newtheorem{thm}{Theorem}[section]
\newtheorem{lem}[thm]{Lemma}
\newtheorem{prop}[thm]{Proposition}
\newtheorem{cor}[thm]{Corollary}
\theoremstyle{definition}
\newtheorem{dfn}[thm]{Definition}
\newtheorem*{ac}{Acknowledgments}  
\newtheorem{ques}[thm]{Question}
\newtheorem{rem}[thm]{Remark}

\newtheorem{chunk}[thm]{}
\theoremstyle{remark}

\numberwithin{equation}{thm}
\def\ann{\operatorname{ann}}
\def\Min{\operatorname{Min}}
\def\ipd{\operatorname{IPD}}
\def\iid{\operatorname{IID}}
\def\igd{\operatorname{IGD}}
\def\gdim{\operatorname{G-dim}}

\def\depth{\operatorname{depth}}

\def\End{\operatorname{End}}
\def\Ext{\operatorname{Ext}}

\def\syz{\mathsf{\Omega}}

\def\H{\mathrm{H}}
\def\Hom{\operatorname{Hom}}
\def\id{\mathrm{id}}

\def\m{\mathfrak{m}}
\def\Max{\operatorname{Max}}

\def\ng{\operatorname{NonGor}}
\def\Ass{\operatorname{Ass}}  
\def \coker {\operatorname{coker}}
\def \H {\operatorname{H}}
\def\p{\mathfrak{p}}

\def\sing{\operatorname{Sing}}

\def\spec{\operatorname{Spec}}

\def\supp{\operatorname{Supp}}
\def\syz{\mathrm{\Omega}}

\def\Tor{\operatorname{Tor}}
\def\reg{\operatorname{Reg}}

\def\Tr{\operatorname{Tr}}
\def\tr{\operatorname{tr}}   

\def\V{\operatorname{V}}

\def \c{\mathfrak c} 
\def \pd{\operatorname{pd}}
\def\id{\operatorname{id}}  
\def \I {\operatorname{I}}  
\def \cl {\operatorname{cl}}

\def\ord{\operatorname{ord}}
\def \im {\operatorname{Im}}
\def \grade{\operatorname{grade}}  
\def \Im {\operatorname{Im}}  
\def \Ker {\operatorname{Ker}} 
  
\begin{document}
%\allowdisplaybreaks
\title{On Containment of trace ideals in ideals of finite homological dimension}
\author{Souvik Dey}
\address{S. Dey, Department of Algebra, Charles University, Faculty of Mathematics and Physics, Sokolovska
83, 186 75, Praha, Czech Republic}
\email{ souvik.dey@matfyz.cuni.cz} 
\urladdr{\url{https://orcid.org/0000-0001-8265-3301}}  

\author{Monalisa Dutta}
\address{M. Dutta, Department of Mathematics, University of Kansas, KS 66045-7523, Lawrence, USA} \email{ m819d903@ku.edu} 

\subjclass[2020]{13C13, 13C14, 13D02, 13D05, 13D07}
\keywords{Trace ideal, Closure operation, big Cohen--Macaulay module, projective dimension, injective dimension, singular locus.}  

\thanks{Souvik Dey was partly supported by the Charles University Research Center program No.UNCE/24/SCI/022 and a grant GA \v{C}R 23-05148S from the Czech Science Foundation.}  

\maketitle

\begin{abstract} 
Motivated by recent result of P\'erez and R.G. on equality of test ideal of module closure operation and trace ideal, and the well-known result by Smith that parameter test ideal cannot be contained in parameter ideals, we study the obstruction of containment of trace ideals in ideals of finite projective (or injective) dimension. One of our results says that the trace ideal of any big Cohen--Macaulay module over a Gorenstein complete local domain cannot be contained in any ideal of finite projective dimension, thereby  generalizing Smith's result in this case. As consequences of our results, we give upper bounds on $\m$-adic order of trace ideals of certain modules over local Cohen--Macaulay rings. We also prove analogous results for ideal of entries of maps in a free resolution of  modules.
\end{abstract}

\section{Introduction}

This work is motivated by two results, one by K. E. Smith (see \cite[Proposition 6.1]{smith}) stating that the parameter test ideal for tight closure cannot contained in any ideal generated by any system of parameters, and another is by F. P\'erez and R. R.G. (see \cite[Theorem 3.12]{pr}) which says that the test ideal associated to a module closure operation is a trace ideal. To illustrate  these  in more detail, let us recall some definitions.

Throughout, $R$ will denote a commutative Noetherian ring.  

\begin{dfn}(see \cite[Definition 2.1]{pr}) 
 A closure operation $\operatorname{cl}$ on a ring $R$ is a map, which to each pair of modules $N \subseteq M$ assigns a submodule $N^{\cl}_M$ of $M$ satisfying the following properties:
\begin{enumerate}
\item $N \subseteq N^{\cl}_M$,
\item $(N^{\cl}_M)^{\cl}_M = N^{\cl}_M$, and
\item $N^{\cl}_M\subseteq N'^{\cl}_M$ for $R$-modules $N \subseteq N' \subseteq M$
\end{enumerate}
\end{dfn}

% and is defined as follows:
% \begin{dfn}(see \cite{hoclec})
% Let $R$ be a ring of characteristic $p > 0$. The Frobenius endomorphism is the map $F : R \to R$ sending $r \mapsto r^p$. The $e$-th iteration of
% $F$ sends $r \mapsto r^{p^e}.$ Let $F^e_*(R)$ denote the abelian group $R$, viewed as an $R$-module via the Frobenius endomorphism.  Let  $N \subseteq M$ be $R$-modules.  Then the tight closure of $N$ in $M$ is denoted by $N^*M$ and defined as follows: $u \in N^*M$ if there is some $c \in R$ not in any
% minimal prime of $R$ such that $c \otimes u \in \operatorname{Im}(F^e_*(R)\otimes_R N\to F^e_*(R)\otimes_R M)$ for all $e\gg0$.
% \end{dfn}

One very important example of a closure operation is tight closure for modules over rings of prime characteristic. This is usually denoted by $N^*_M$ for a pair of modules $N\subseteq M$ and is defined as follows:

\begin{dfn}(see \cite{hoclec})
Let $R$ be a ring of characteristic $p > 0$. The Frobenius endomorphism is the map $F : R \to R$ sending $r \mapsto r^p$. The $e$-th iteration of
$F$, denoted by $F^e$, sends $r \mapsto r^{p^e}.$ Let $F^e_*(R)$ denote the abelian group $R$, viewed as an $R$-module via the $e$-th iteration of Frobenius endomorphism. Let $N \subseteq M$ be $R$-modules. Then the tight closure of $N$ in $M$ is denoted by $N^*_M$ and defined as follows: $u \in N^*_M$ if there is some $c \in R$, not in any  minimal prime of $R$, such that $c \otimes u \in \operatorname{Im}(F^e_*(R)\otimes_R N\to F^e_*(R)\otimes_R M)$ for all $e\gg0$.  
\end{dfn}

 Another important closure operation is the module closure associated to an $R$-module $B$ defined as follows:

\begin{dfn}(see \cite[Definition  2.4]{pr}) Let $R$ be a ring. Given an $R$-module $B$ (not necessarily finitely-generated), the module closure operation $\cl_B$ on $R$ is defined as follows: $$u\in N^{\cl_B}_M\text{ if for all }b\in B,\text{ }b\otimes u\in \operatorname{Im}(B \otimes N \to B \otimes M)$$
for any pair of $R$-modules $N \subseteq M$ and $u \in M$.
\end{dfn}

The way these two closure operations are related is via the following result  by M. Hochster:
\begin{thm}(\cite[Theorem 11.1]{hoc}) 
Let $R$ be a complete local domain of characteristic
$p > 0$, and let $N \subseteq M$ be finitely-generated $R$-modules. Then $N^*_M$, the tight closure of $N$ in $M$, is equal to the set of elements $u \in M$ that are in $N^{\cl_B}_M$ for some big Cohen-Macaulay algebra $B$.
\end{thm}

It also follows from the works of R. R.G., and G. Dietz that for any complete local domain of prime characteristic, there exists a big Cohen-Macaulay module such that the tight closure on finitely generated modules is contained in the closure operation induced by that big Cohen--Macaulay module

\begin{thm}(\cite[Example 5.4, Lemma 5.5, Proposition 5.6]{diet}, \cite[Definition 2.9, Theorem 5.1]{rrg})\label{bigd} Let $(R,\m)$ be a complete local domain of positive characteristic. Then, there exists a big Cohen--Macaulay $R$-module $B$ such that $N_M^*\subseteq N_M^{\cl_B}$ for any finitely generated $R$-modules $N\subseteq M$. 
\end{thm}

To any closure operation, one can associate a test ideal and a parameter test ideal defined respectively as follows:
\begin{dfn}\label{test}(see \cite[Definition 3.1]{pr}) Let $R$ be a ring and $\cl$ be a closure operation on $R$-modules. The big test ideal
of $R$ associated to $\cl$ is defined by $\tau_{\cl}(R) = \cap_{N\subseteq M}
\left(N :_R N^{\cl}_M\right)$, where the intersection runs over any (not necessarily finitely-generated) $R$-modules $N, M$.
\end{dfn}  

% \begin{dfn}\label{testfam}(see \cite[Definition 4.1]{pr})
% Let $\mathcal B$ be a family of $R$-modules, not necessarily finitely-generated. The test ideal associated to $\mathcal B$ is defined as $\tau_{\mathcal B}(R) = \cap_{B\in\mathcal B}
% \tau_B(R)$.
% \end{dfn}

\begin{dfn}(\cite[Definition 4.3]{smith})\label{para}
Let $R$ be a ring and $\cl$ be a closure operation on $R$-modules. The parameter test ideal of $R$ is the ideal of all elements $c \in R$ such that $cI^{\cl}_R \subseteq I$ for all parameter ideals $I$ of $R$, i.e., $\tau_{\cl}^{\text{para}}(R) = \cap_{I\subseteq R}
\left(I :_R I^{\cl}_R\right)$, where the intersection runs over all parameter ideals $I$ of $R$. 
\end{dfn} 

With these definitions and notations in hand, we can now mention the results of K.E. Smith and the one by F. P\'erez and R. R.G, respectively, as follows:

\begin{thm}(\cite[Proposition 6.1]{smith})\label{smit}
Let $R$ be a local Cohen--Macaulay ring. Then $\tau_*^{\text{para}}(R)\not\subseteq J$ for any parameter ideal $J$ of $R$.   
\end{thm}

\begin{thm}(\cite[Theorem 3.12]{pr})\label{testtrace}
Let $R$ be a local ring and $\cl = \cl_B$ for some $R$-module $B$. If $B$ is a finitely generated $R$-module or $R$ is complete, then $\tau_{\cl}(R) = \tr_R(B)$, where $\tr_R(B)= \sum_{\phi\in \Hom_R(B,R)}\phi(B)$.  
\end{thm}

\begin{rem}\label{nww} Let $R$ be a complete local domain of positive characteristic. It is an immediate consequence of Theorem \ref{bigd}, Definitions \ref{test}, \ref{para}, and Theorem \ref{testtrace} that there exists a big Cohen--Macaulay $R$-module $B$ such that $\tr_R(B)$ is contained in the parameter test ideal $\tau_*^{para}(R)$ of tight closure.  
\end{rem}

 Theorem \ref{smit}, \ref{testtrace} and Remark \ref{nww} naturally give rise to the following question: 
 
 % and the observation that $\tau_{\cl}(R)$ is always a subset of $\tau_{\cl}^{\text{para}}(R)$,   

\begin{ques}\label{mn}
Let $M$ be a module over a local Cohen--Macaulay ring $R$. When is $\tr_R(M)$ not contained in a parameter ideal?
\end{ques}

One of our results towards answering Question \ref{mn} is the following, where $\pd_R(-)$ stands for projective dimension.

\begin{thm}[Theorem \ref{big}]\label{1.10}

Let $R$ be a semi-local ring , $I\subseteq J(R)$ be an ideal of $R$, and $M$ be an $R$-module. Let $\pd_R(I)\leq h<\infty$ and $\tr_R(M)\subseteq I$. Assume either $R$ is local, or $I$ contains a non-zero-divisor. If $\Ext_R^{1\leq i\leq h}(M,R)=0$, then $M^*=0$, i.e., $\tr_R(M)=0$.
\end{thm}

If one moreover assumes $M$ is finitely generated, then a variation of Theorem \ref{1.10} can be provided without assuming the ring is semi-local, see Theorem \ref{hah}.  

For big Cohen--Macaulay modules over complete local Cohen--Macaulay rings, we prove the following, where $\tau_{\omega}(M) :=\sum_{f\in \Hom_R(M,\omega)}f(M)$ (see \ref{tau}).   

\begin{thm}[Theorem \ref{bigmac}]\label{newIntro} Let $(R,\m)$ be a complete local Cohen--Macaulay  ring. Let $M$ be a big Cohen--Macaulay $R$-module and $I\subseteq \m$ be an ideal of $R$ with $\pd_R I<\infty$. If $\omega$ is a dualizing module of $R$, then $\tau_{\omega}(M)\nsubseteq I\omega$. In particular, if $R$ is Gorenstein, then $\tr_R(M)\nsubseteq I$. 
    
\end{thm}

In view of Remark \ref{nww}, the complete local Gorenstein domain case of Smith's result \cite[Proposition 6.1]{smith} is an immediate consequence of Theorem \ref{newIntro}.

As another consequence, we are able prove an injective dimension version of Theorem \ref{1.10}. 

\begin{thm}[Remark \ref{3.22}]\label{111}  Let $(R,\m)$ be a complete local Cohen--Macaulay  ring  admitting a canonical module. Let $M$ be a big Cohen--Macaulay $R$-module. Let $0\neq J (\subseteq \m)$ be an  ideal of finite injective dimension. Then $R$ is generically Gorenstein, there exists a  canonical module $\omega$ satisfying $R\subseteq \omega\subseteq Q(R)$, and for any such canonical module,  $\tau_{\omega}(M)\nsubseteq J$. 
\end{thm}

Under certain vanishing hypotheses on  Ext modules, we are able to relate the trace and ideal of entries in a free resolution of certain modules in Section \ref{trid}. As a consequence of this and Theorem \ref{1.10}, we obtain the following two results, where
$\ord(I):=\sup \{n\in\mathbb{Z}_{\ge 0}\text{ }|\text{ } I\subseteq \m^n\}$ and $\I_1(\partial_j)$ will stand for the ideal generated by the entries of a matrix representation of the map $\partial_j$ in a free resolution $(F_i,\partial_i)_{i=0}^\infty$ of $M$.

\begin{thm}[Corollary \ref{44}]\label{133}   Let $I$ be an ideal of $R$ such that $ I$ is contained in the Jacobson radical of $R$ and $\pd_R (I)\leq h<\infty$. Let $M$ be a finitely generated $R$-module and $j\geq 1$ be an integer such that $\Ext_R^{j\leq i \leq j+h}(M, R)=0$. If $(F_i,\partial_i)_{i=0}^\infty$ is a free resolution of $M$, and $\I_1(\partial_j)\subseteq I$, then $\pd_R M<j$.  
\end{thm}

\begin{thm}[Corollary \ref{62}]\label{162} Let $(R,\m)$ be a local Cohen--Macaulay ring of depth $t\geq 1$, with infinite residue field, and $M$ be a non-free finitely generated $R$-module such that $R$ is not a direct summand of $\syz_R M$, and $\Ext^{1\le i \le t}_R(M,R)=0$. If $(F_i,\partial_i)_{i=0}^\infty$ is a minimal free resolution of $M$, then $1\le \ord\left(\I_1(\partial_1)\right)\le e(R)-\mu(\m)+t$, where $e(R)$ is the Hilbert-Samuel multiplicity. In particular, equality holds throughout when $R$ has minimal multiplicity.  
\end{thm}  

In Section \ref{sec7}, we consider the question that if $J$ is an ideal defining the singular locus of $R$, then when is $J$ not contained in any ideal of finite projective or injective dimension.  We prove certain results in this direction, for instance Propositions \ref{sin}, \ref{intg}, and also relate them to obstruction of containment of trace ideals. One of our sample results is the following.   

\begin{thm}[Theorem \ref{75}] Let $M$ be an $R$-module of full support and $\depth_{R_{\p}} M_{\p} \ge \depth R_{\p} $ for all $\p \in \reg(R)$. Let $I, L$ be ideals of $R$, where $I$ is a proper  ideal. Then the following hold true:

\begin{enumerate}[\rm(1)]

    \item If $R$ is local, $I$ is a radical ideal, $n\ge 1$ is an integer such that $\tr_R (M)\subseteq I^{(n)}L$, then $\pd_R I^{(n)}L=\infty=\id_R I^{(n)}L$.

    \item If $I$ is integrally closed and $\tr_R (M)\subseteq I$, then $\pd_{R_{\p}}(I_{\p})=\infty=\id_{R_{\p}}(I_{\p})$ for all $\p\in \Min(R/I)$. In particular, $\pd_R I=\infty=\id_R(I)$.    
\end{enumerate}
    
\end{thm}

The organization of the paper is as follows: In Section \ref{prelim} we record basic notations, conventions and some preparatory results that are used for the rest of the paper. Section  \ref{secmain} is devoted to proving Theorems \ref{1.10}, Theorem \ref{newIntro} and \ref{111}, a crucial ingredient for these is the main technical result Theorem \ref{4}. Section \ref{trid} contains Theorem \ref{133} among some other results relating trace ideal and ideal of entries of maps in a free resolution of modules. Section \ref{ortr} is devoted to obtaining some upper bounds of the $\m$-adic order of certain trace ideals over local Cohen--Macaulay rings $(R,\m)$. As a consequence of this, Section \ref{minfreeord} records some upper bounds of the $\m$-adic order of ideal of entries of maps in free resolutions, in particular Theorem \ref{162} is deduced in this section.

\begin{ac} We are immensely thankful to our PhD advisor Hailong Dao for urging us to look for a trace ideal version of \cite[Proposition 6.1]{smith}. We are also grateful to Linquan Ma for suggesting us to extend some of the results in a previous version of this paper to modules which are not necessarily finitely generated, for explaining Lemma \ref{bigcm} to us, and for asking us about obstruction of containment of ideals defining singular locus in ideals of finite homological dimension, which prompted us to write Section \ref{sec7}. 
\end{ac}

\section{preliminaries}\label{prelim}   
In this section, we fix some basic setup and notations, that will be used throughout the rest of the paper without possible further references. We will also record some preliminary observations for further use in the latter sections.

Throughout, $R$ will denote a commutative Noetherian  ring, with total ring of fractions $Q(R)$, Jacobson radical $J(R)$. $\overline R$ will stand for the integral closure of $R$ in $Q(R)$. When $R$ admits a canonical module (necessarily $R$ is Cohen--Macaulay in such a case), we will use the terminologies \textit{dualizing module} and \textit{canonical module} interchangeably. In this case, we will denote such a module by either $D$ or $\omega_R $ (often dropping $R$ and just writing $\omega$ when the ring $R$ in context is clear). 

\begin{chunk} For $R$-submodules $M,N$ of $Q(R)$, $(M:N)$ will stand for $\{x\in Q(R)|xN\subseteq M\}$, and this is clearly an $R$-submodule of $Q(R)$. 
\end{chunk} 

\begin{chunk}\label{tau}
For $R$-modules $X,M$, let $\tau_X(M) :=\sum_{f\in \Hom_R(M,X)}f(M)$ (see \cite[Section 2]{gsi}). We will denote $\tau_R(M)$ by $\tr_R(M)$, and when the ring in question is clear, we will drop the suffix $R$ from $\tr_R(-)$. It is clear that if $\Hom_R(M,X)=0$, then $\tau_X(M)=0$. Conversely, if $\tau_X(M)=0$, then $\Hom_R(M,X)=\Hom_R(M,\tau_X(M))=0$. In particular, $\tr_R(M)=0$ if and only if $M^*=0$.  
\end{chunk}   

\begin{chunk}\label{conductor} Given a birational extension $R\subseteq S$, we define the conductor ideal of $S$ in $R$, denoted by $c_R(S)$, to be $(R:S)=\{r\in Q(R):rS\subseteq R\}$. Since $1\in S$, so $c_R(S)=(R:_R S)$. Notice that, $c_R(S)$ is an ideal of $R$. Moreover, since $S\cdot c_R(S) \cdot S\subseteq S\cdot c_R(S)\subseteq R$, so $S\cdot c_R(S)\subseteq (R:S)=c_R(S)$. Hence $c_R(S)$ is always an ideal of $S$ as well. Since $c_R(S)=(R:S)=(R:S)S=\tr_R(S)$ (\cite[Proposition 2.4(2)]{tk}), hence $c_R(S)$ is always a trace ideal of $R$.  

Now, $S$ is module finite over $R$ if and only if the conductor ideal $ c_R(S)$ contains a non-zero-divisor of $R$ (i.e., $ c_R(S)\in \mathcal F_R$ in the notation of \cite{gsi}). So, if $S$ is a finite birational extension of $R$, then $(R:c_R(S))=(c_R(S):c_R(S))$ by \cite[Proposition 2.4(3)]{tk}.  We will usually denote $c_R(\overline R)$ by just $\mathfrak c$. 
\end{chunk}

\begin{chunk}

For a finitely generated $R$-module $M$, consider a projective presentation $P_1 \stackrel{\eta}{\to} P_0 \to M \to 0$ of $M$ by finitely generated projective modules. The  $($Auslander$)$ transpose of $M$ is defined to be  $\Tr(M) := \coker(\eta^*)$. 
    By the definition of $\Tr(M)$, there is an exact sequence
\begin{equation}\label{eqn:es-M*-Tr-M}
    0 \to M^* \to P_0^* \stackrel{\eta^*}{\longrightarrow} P_1^* \to \Tr(M)\to 0.
\end{equation}
It is well known that $\Tr M$ is well-defined up to projective summands, and that $M$ and $\Tr(\Tr M)$ are stably isomorphic (i.e., these two modules are isomorphic up to projective summands).  For further details on this topic, see \cite{mas}.

% For an $R$-module $M$, consider a projective presentation $F_1 \stackrel{\eta}{\to} F_0 \to M \to 0$ of $M$. The $($first$)$ syzygy and $($Auslander$)$ transpose of $M$ are defined to be $\syz_R(M) := \Im(\eta) $ and $\Tr(M) := \coker(\eta^*)$ respectively. When the base ring is clear, we drop the subscript from $\syz_R(M)$. Set $\syz^1(M)=\syz(M)$, and inductively $\syz^n(M)=\syz^1(\syz^{n-1}(M))$ for all $n\ge 2$. These modules are uniquely determined by $M$ up to projective summands.
%     % Let $\Tr(M)$ denote the Auslander transpose of $M$. If $F_1 \stackrel{\eta}{\to} F_0 \to M \to 0$ is a free presentation of $M$, then b
%     By the definition of $\Tr(M)$, there is an exact sequence
% \begin{equation}\label{eqn:es-M*-Tr-M}
%     0 \to M^* \to F_0^* \stackrel{\eta^*}{\longrightarrow} F_1^* \to \Tr(M)\to 0.
% \end{equation}
% It is well known that $M$ and $\Tr(\Tr M)$ are stably isomorphic (i.e., these two modules are isomorphic up to projective summands).  When $R$ is local, $\syz^n(M)$ is only defined by considering minimal free resolution of $M$. 
\end{chunk}  

% Recall from () that for any $R$-modules $M,N$ we have, $\underline{\Hom}_R(M,N)\cong \Tor^R_1(\Tr M, N)$.  
Now we record two observations regarding morphisms between modules that can be given by multiplication by some matrix. We will write elements of a free $R$-module $R^{\oplus s}$ as column vectors.

For an $R$-module $M$ and for any integer $s>0$, the modules $R^{\oplus s}\otimes_R M$ and $M^{\oplus s}$ are isomorphic via the map $\begin{pmatrix}
r_1\\
\vdots\\
r_s
\end{pmatrix}\otimes m\mapsto \begin{pmatrix}
r_1m\\
\vdots\\
r_sm
\end{pmatrix}$. Throughout this article, we will use this identification without possibly referencing further.

We recall that any $R$-linear map $f$ of free modules has a matrix representation, and the representation does not depend on the choice of basis of the free modules.

\begin{lem}\label{1} Let $f:R^{\oplus s}\to R^{\oplus t}$ be an $R$-linear map, whose matrix representation has entries in an ideal $I$ of $R$. Then for any $R$-module $N$, the induced map $f\otimes \id_N: N^{\oplus s}\to N^{\oplus t}$ can be given by a matrix multiplication with entries in $I$.
\end{lem}

\begin{proof}  
Let $A=[(a_{ij})]_{1\leq i \leq t, 1\leq j \leq s}$ be the matrix representation of $f$, so $a_{ij}\in I$.  Now let, $\begin{pmatrix}
n_1\\
\vdots\\
n_s
\end{pmatrix}$ be an arbitrary element of $N^{\oplus s}$. Then $\begin{pmatrix}
n_1\\
\vdots\\
n_s
\end{pmatrix}$ can be identified with the element $\sum_{j=1}^s (e_j\otimes n_j)\in R^{\oplus s}\otimes N$, where $e_j$ is the $j$-th standard basis vector in $R^{\oplus s}$. Then we have $(f\otimes \id_N)\begin{pmatrix}
n_1\\
\vdots\\
n_s
\end{pmatrix}=\sum_{j=1}^s (f(e_j)\otimes n_j)=\sum_{j=1}^s (Ae_j\otimes n_j)=\sum_{j=1}^s \left(\begin{pmatrix}
a_{1j}\\
\vdots\\
a_{tj}
\end{pmatrix}\otimes n_j\right)=\sum_{j=1}^s \begin{pmatrix}
a_{1j}n_j\\
\vdots\\
a_{tj}n_j
\end{pmatrix}=\begin{pmatrix}
\sum_{j=1}^sa_{1j}n_j\\
\vdots\\
\sum_{j=1}^sa_{tj}n_j
\end{pmatrix}=A\begin{pmatrix}
n_1\\
\vdots\\
n_s
\end{pmatrix}$, which shows that the map $f\otimes \id_N: N^{\oplus s}\to N^{\oplus t}$ can be given by the same matrix multiplication $A$, whose entries lie in $I$.
\end{proof}

\begin{lem}\label{2} Let $N$ be an $R$-module and let, $g:N^{\oplus s}\to N$ be an $R$-linear map which can be given by multiplication by a matrix with entries in an ideal $I$ of $R$. Then for any $R$-module $M$, the map $\Hom_R(M,g):\Hom_R(M,N)^{\oplus s}\to \Hom_R(M,N)$ can be given by a matrix multiplication with entries in $I$.
\end{lem}

\begin{proof}
Let $g:N^{\oplus s}\to N$ be given by multiplication by a matrix $A=(r_1,\cdots ,r_s)$, where $r_i\in I$ for all $i=1,\cdots,s$. Hence $g(n_1,\cdots,n_s)=r_1n_1+\cdots+r_sn_s$ for all $(n_1,\cdots,n_s)\in N^{\oplus s}$. Now note that, the map $\Hom_R(M,g):\Hom_R(M,N)^{\oplus s}\to \Hom_R(M,N)$ is defined by the following composition:
$$\Hom_R(M,N)^{\oplus s}\xrightarrow{\psi_1}\Hom_R(M,N^{\oplus s})\xrightarrow{\psi_2} \Hom_R(M,N)$$ where $\psi_1$ is defined by $\psi_1(f_1,\cdots,f_s)=f_1\oplus\cdots\oplus f_s$, where $(f_1\oplus\cdots\oplus f_s)(m)=(f_1(m),\cdots,f_s(m))$ and $\psi_2$ is defined by $\psi_2(f)=g\circ f$. Hence for all $m\in M$ we have,
\begin{align*}
(\Hom_R(M,g)(f_1,\cdots,f_s))(m)=((\psi_2\circ\psi_1)(f_1,\cdots,f_s))(m)&=(\psi_2(f_1\oplus\cdots\oplus f_s))(m)\\
&=(g\circ(f_1\oplus\cdots\oplus f_s))(m)=g(f_1(m),\cdots,f_s(m))\\
&=r_1f_1(m)+\cdots+r_sf_s(m)=(r_1f_1+\cdots+r_sf_s)(m)
\end{align*}
Thus $$\Hom_R(M,g)(f_1,\cdots,f_s)=r_1f_1+\cdots+r_sf_s=(r_1,\cdots,r_s)\begin{pmatrix}
f_1\\
\vdots\\
f_s
\end{pmatrix}=A\begin{pmatrix}
f_1\\
\vdots\\
f_s
\end{pmatrix}$$
This shows that the map $\Hom_R(M,g):\Hom_R(M,N)^{\oplus s}\to \Hom_R(M,N)$ can be given by the same matrix multiplication $A$, whose entries lie in $I$.
\end{proof}

We conclude this section with the following observation that will be crucial in proving one of our main technical results.

\begin{lem}\label{3} Let $X$ be a submodule of an $R$-module $N$. Then, for any $R$-module $M$, we have $$\ker \left(\Hom_R(M,N)\xrightarrow{\Hom_R(M,\pi)} \Hom_R(M,N/X)\right)=\Hom_R(M,X)$$  
\end{lem}

\begin{proof} This follows from the fact that $0\to \Hom_R(M,X) \xrightarrow{i} \Hom_R(M,N)\xrightarrow{\Hom_R(M,\pi)} \Hom_R(M,N/X)$ is exact.  
\end{proof} 

\begin{lem}\label{nak} Let $M, N$ be $R$-modules such that $N$ is finitely generated. If $\Hom_R(M,N)=J(R)\Hom_R(M,N)$, then $\Hom_R(M,N)=0$. 
\end{lem}

\begin{proof}
Assume that, $\Hom_R(M,N)\neq 0$. Then there exists $f\in\Hom_R(M,N)$ and $x\in M$ such that $f(x)\neq 0$. Define $\phi:\Hom_R(M,N)\to N$ by $\phi(\psi)=\psi(x)$. Let $L:=\Im(\phi)$ and $K:=\Ker(\phi)$. Since $0\neq f(x)\in L$, so $L\neq 0$. Now we have, $\frac{\Hom_R(M,N)}{K}\cong L$, which implies $\frac{\Hom_R(M,N)}{K}\otimes_R\frac{R}{J(R)}\cong L\otimes_R\frac{R}{J(R)}$. This implies $\frac{\Hom_R(M,N)}{K+J(R)\Hom_R(M,N)}\cong\frac{L}{J(R)L}$. Since $\Hom_R(M,N)=J(R)\Hom_R(M,N)$, so $\frac{\Hom_R(M,N)}{K+J(R)\Hom_R(M,N)}=0$. Hence $L=J(R)L$. Now $L$ is a submodule of the finitely generated $R$-module $N$ and $R$ is Noetherian. Hence $L$ is finitely generated, so by Nakayama's lemma $L=J(R)L$ implies $L=0$, which is a contradiction. Therefore, $\Hom_R(M,N)=0$.
\end{proof}

\section{when are trace ideals not contained in ideals of finite projective or injective dimension?}\label{secmain} 

In this section, we prove our main results about obstructions of containment of trace ideals inside ideals of finite projective or injective dimension. Most of such results will be deduced from the main technical result Theorem \ref{4}. 

\begin{thm}\label{4} Let $I$ be an ideal of  $R$. Let $N$ be a finitely generated $R$-module such that $\Tor^R_{1}(N,R/I)=0$. Let $M$ be an $R$-module such that one of the following conditions hold:

\begin{enumerate}[\rm(1)]
    \item $\Tor^R_{2}(N,R/I)=0$, and $\Ext^1_R(M,N\otimes_R \syz_R I)=0$, where $\syz_R I$ can be taken to be the kernel of some $R$-linear map $R^{\oplus n}\to R$ whose image is $I$.   

    \item $M$ is finitely generated and $\Tor^R_1(\Tr M, IN)=0$. 
\end{enumerate}

 Then we have $\Hom_R(M,IN)= I\Hom_R(M,N)$. So, if moreover $\tau_N(M)\subseteq IN$, and $I\subseteq J(R)$, then $\Hom_R(M,N)=0$.    
\end{thm}  

\begin{proof} Clearly, $I\Hom_R(M,N)\subseteq \Hom_R(M,IN)$, so we only have to show $\Hom_R(M,IN)\subseteq I\Hom_R(M,N)$. 
We have an exact sequence $0\to \syz_R I\to R^{\oplus n}\xrightarrow{f} R\xrightarrow{\pi} R/I\to 0$ with $f$ having entries in $I$ (since $\operatorname{Im}(f)=\operatorname{ker}(\pi)=I$ and a matrix representation of $f$ is given by $(f(e_1)\cdots f(e_n)))$. This breaks apart into short exact sequences $\gamma:\text{ }0\to \syz_R I\to R^{\oplus n} \xrightarrow{g} I\to 0$ and $\sigma:\text{ }0\to I \xrightarrow{h} R\xrightarrow{\pi} R/I\to 0$, with $f=h\circ g$. 

Now for both cases (1) and (2), we first show that we get an exact sequence 
\begin{align}\label{exact}
\Hom_R(M,N^{\oplus n})\xrightarrow{\Hom_R(M,\widetilde f)} \Hom_R(M,N)\xrightarrow{\Hom_R(M,\pi)} \Hom_R(M,N/IN)
\end{align}

(1)  Since $\Tor^R_1(N,R/I)=0=\Tor^R_1(N,I)$, so tensoring $\gamma$ and $\sigma$ with $N$ we get exact sequences $0\to N\otimes_R \syz_R I\to N^{\oplus n} \xrightarrow{g\otimes \id_N} I\otimes_R N\to 0$ and $0\to I \otimes_R N \xrightarrow{h\otimes \id_N} N\xrightarrow{\pi} N/IN\to 0$, with $f\otimes \id_N=(h\otimes \id_N)\circ(g\otimes \id_N)$ and $\pi$ is still the canonical surjection. Hence we get an exact sequence 
$0\to N\otimes_R \syz_RI \to N^{\oplus n}\xrightarrow{\widetilde f}N\xrightarrow{\pi} N/IN \to 0$, where $\widetilde f:=f\otimes \id_N$ is given by a matrix multiplication with entries in $I$ (by Lemma \ref{1}). As $\im \widetilde f=\ker \pi=IN$, so we can break it apart into two short exact sequences $\alpha:\text{ }0\to N\otimes_R \syz_RI \to N^{\oplus n} \xrightarrow{\theta} IN \to 0$ and $\beta:\text{ }0\to IN \xrightarrow{\phi}  N \xrightarrow{\pi}N/IN\to 0$, with $\widetilde f=\phi\circ \theta$. By applying $\Hom_R(M,-)$ to the sequences $\beta$ and $\alpha$ we get exact sequences $$0\to \Hom_R(M,IN)\xrightarrow{\Hom_R(M,\phi)}\Hom_R(M,N)\xrightarrow{\Hom_R(M,\pi)} \Hom_R(M,N/IN)$$  \begin{align*}
&0\to \Hom_R(M,N\otimes_R \syz_RI)\to \Hom_R(M,N^{\oplus n})\xrightarrow{\Hom_R(M,\theta)} \Hom_R(M,IN)\to 0
\end{align*}
Since $\Hom_R(M,\widetilde f)=\Hom_R(M,\phi)\circ \Hom_R(M,\theta)$, hence we get an exact sequence $$0\to \Hom_R(M,N\otimes_R \syz_RI)\to \Hom_R(M,N^{\oplus n})\xrightarrow{\Hom_R(M,\widetilde f)} \Hom_R(M,N)\xrightarrow{\Hom_R(M,\pi)} \Hom_R(M,N/IN)$$

(2)  Since $\Tor^R_1(N,R/I)=0$, so tensoring $\gamma$ and $\sigma$ with $N$ we get exact sequences $0\to L\to N^{\oplus n} \xrightarrow{g\otimes \id_N} I\otimes_R N\to 0$ and $0\to I \otimes_R N \xrightarrow{h\otimes \id_N} N\xrightarrow{\pi} N/IN\to 0$, where $L$ is some $R$-module, and $f\otimes \id_N=(h\otimes \id_N)\circ(g\otimes \id_N)$ and $\pi$ is still the canonical surjection. Hence we get an exact sequence 
$0\to L \to N^{\oplus n}\xrightarrow{\widetilde f}N\xrightarrow{\pi} N/IN \to 0$, where $\widetilde f:=f\otimes \id_N$ is given by a matrix multiplication with entries in $I$ (by Lemma \ref{1}). As $\im \widetilde f=\ker \pi=IN$, so we can break it apart into two short exact sequences $\alpha:\text{ }0\to L \to N^{\oplus n} \xrightarrow{\theta} IN \to 0$ and $\beta:\text{ }0\to IN \xrightarrow{\phi}  N \xrightarrow{\pi}N/IN\to 0$, with $\widetilde f=\phi\circ \theta$. By applying $\Hom_R(M,-)$ to the sequence $\beta$ we get exact sequence $$0\to \Hom_R(M,IN)\xrightarrow{\Hom_R(M,\phi)}\Hom_R(M,N)\xrightarrow{\Hom_R(M,\pi)} \Hom_R(M,N/IN)$$ By using \cite[Lemma 3]{blms} and \cite[Proposition 12.9]{lw}, from the sequence $\alpha$ we get an exact sequence \begin{align*}
&0\to \Hom_R(M,L)\to \Hom_R(M,N^{\oplus n})\xrightarrow{\Hom_R(M,\theta)} \Hom_R(M,IN)\to 0
\end{align*}

Since $\Hom_R(M,\widetilde f)=\Hom_R(M,\phi)\circ \Hom_R(M,\theta)$, hence we get an exact sequence $$0\to \Hom_R(M,L)\to \Hom_R(M,N^{\oplus n})\xrightarrow{\Hom_R(M,\widetilde f)} \Hom_R(M,N)\xrightarrow{\Hom_R(M,\pi)} \Hom_R(M,N/IN)$$

This finally proves the existence of the exact sequence \ref{exact} for both cases (1) and (2).

Notice that, $\Hom_R(M,\widetilde f)$ can be given by a matrix multiplication with entries in $I$ (by Lemma \ref{2}). Thus $\im\left(\Hom_R(M,N^{\oplus n})\xrightarrow{\Hom_R(M,\widetilde f)} \Hom_R(M,N) \right) \subseteq I\Hom_R(M,N)$.
However, $$\im\left(\Hom_R(M,N^{\oplus n})\xrightarrow{\Hom_R(M,\widetilde f)} \Hom_R(M,N) \right)=\ker \left(\Hom_R(M,N)\xrightarrow{\Hom_R(M,\pi)} \Hom_R(M,N/IN)\right)$$ which is $\Hom_R(M,IN)$ by Lemma \ref{3}.  Thus, $\Hom_R(M,IN)\subseteq I\Hom_R(M,N)$. Now if $\tau_N(M)\subseteq IN$, then $\Hom_R(M,N)=\Hom_R(M,\tau_N(M))=\Hom_R(M,IN)$. Hence we get, $\Hom_R(M,N)\subseteq I\Hom_R(M,N)$. So, assuming $I\subseteq J(R)$, by Lemma \ref{nak}, we have $\Hom_R(M,N)=0$.  
\end{proof}

For our first consequence of Theorem \ref{4}, we record the following lemma. 

\begin{lem}\label{extpd} Let  $M,N$ be $R$-modules such that $N$ is finitely generated, and $\pd_R(N)<\infty$ and $\Ext^{1 \le i \le \pd_R(N)+1}_R(M,R)=0$. Then $\Ext^1_R(M,N)=0$.
\end{lem}

\begin{proof} We will prove this by inducting on $n=\pd_R(N)$. If $n=0$, then $N$ is projective. Hence it is a direct summand of a finitely generated free $R$-module. So, the statement is obvious in this case. Next, let the statement be true for $n\le k$. Now let, $n=k+1$ and assume $\Ext^{1 \le i \le k+2}_R(M,R)=0$. Now consider a short exact sequence $0\to \Omega_RN\to F \to N\to 0$, where $F$ is a finitely generated free $R$-module. Since $\pd_R N=k+1\geq 1$, hence $\pd_R(\syz_R N)\leq \pd_R(N)-1=k$ by \cite[Exercise 8.7(ii)]{rotman}. Now apply $\Hom_R(M,-)$ on this short exact sequence and consider the following part of the long exact sequence for all $i$: 
\begin{align}\label{ext}
\cdots\to\Ext^i_R(M,F)\to\Ext^i_R(M,N)\to  \Ext^{i+1}_R(M,\Omega_RN) \to \cdots
\end{align}
%\to\Hom_R(M,\Omega_RN)\to\Hom_R(M,R)^{\oplus \mu(N)}\to\Hom_R(M,N)\to 0 Now since $\pd_R(N)=k+1$, so $\pd_R(\Omega_RN)=k$.

As $F$ is a finitely generated free $R$-module, hence $\Ext^{1 \le i \le k+2}_R(M,R)=0$ implies $\Ext^{1 \le i \le k+2}_R(M,F)=0$.
Also, $\Ext^{1 \le i \le k+2}_R(M,R)=0$ implies that $\Ext^{1 \le i \le k+1}_R(\syz_R M,R)=0$. Hence by applying the induction hypothesis we get that, $0=\Ext^{1}_R(\syz_R M,\Omega_RN)\cong \Ext^2_R(M,\syz_R N)$. Hence from (\ref{ext}) we get that, $\Ext^{1 }_R(M,N)=0$. Thus the statement is true for $n=k+1$ provided it is true for $n\le k$. Hence by induction we get that the statement is true for all $n$.
\end{proof}

\begin{rem}\label{hr2.2} We notice that one does not need $J\neq R$ in \cite[Corollary 2.2(b)]{hr}. This will be used in the following theorem. 
\end{rem}

\begin{thm}\label{big} Let $R$ be a semi-local ring , $I\subseteq J(R)$ be an ideal of $R$, and $M$ be an $R$-module. Let $\pd_R(I)\leq h<\infty$ and $\tr_R(M)\subseteq I$. Assume either $R$ is local, or $I$ contains a non-zero-divisor. If $\Ext_R^{1\leq i\leq h}(M,R)=0$, then $M^*=0$, i.e., $\tr_R(M)=0$.

% Then the following hold true. 
% \begin{enumerate}[\rm(1)]
%     \item Assume either $R$ is local, or $I$ contains a non-zero-divisor. If $\Ext_R^{1\leq i\leq h}(M,R)=0$, then $M^*=0$, i.e., $\tr_R(M)=0$.
%     \item If $(R,\m)$ is local, $\depth R=t$, $M$ is finitely generated, and $\Ext_R^{1\leq i\leq t}(M,R)=0$, then $M=0$.
% \end{enumerate} 

\end{thm} 

\begin{proof}  We may assume $I\neq 0$. First, assume $\pd_R(I)=0$. Then $I$ is projective. If $R$ is local, then $I$ is free. If $I$ contains a non-zero-divisor, then $I_{\m}$ is a non-zero ideal of $R_{\m}$ for every $\m \in \text{Max}\spec(R)$. Hence every $I_{\m}$ is free of rank $1$ over $R_{\m}$, so $I$ is free by \cite[Lemma 1.4.4.]{bh}.  In either case, $I$ is free. Hence $I=fR$ for some non-zero-divisor $f\in J(R)$. Then $\tr_R(M)\subseteq fR$. Put $L:=\frac 1 f \tr_R(M)$. Then $L\subseteq R$ is an $R$-submodule of $R$. Hence $L$ is an ideal of $R$. Now $\tr_R(M)=fL$ and $f\in J(R)$ implies $\tr_R(M)=0$ in view of Remark \ref{hr2.2}. Next, assume  $\pd_R(I)\geq 1$. Consider a $\syz_R I$ in the sense of Theorem \ref{4}(1). Then $\pd_R(\syz_R I)\le \pd_R(I)-1\le h-1$ by \cite[Exercise 8.7(ii)]{rotman}. Now $\Ext_R^{1\leq i\leq {(h-1)}+1}(M,R)=0$ implies $\Ext^1_R(M, \syz_R I)=0$ in view of Lemma \ref{extpd}. Now $M^*=0$ follows from the $N=R$ case of Theorem \ref{4}(1).   

% (2)  Since $R$ is local, so by Auslander-Buchsbaum formula we have $\pd_R(I)\leq t$. Hence by (1) we get $M^*=0$. Thus $\Ext_R^{0\leq i\leq t}(M,R)=0$. 
% If $M\neq 0$, then $\m \in \supp_R(M)$. Hence  \cite[Proposition 1.2.10(a), (e)]{bh} would imply that $\inf \{n: \Ext_R^n(M,R)\neq 0\}\leq \depth R=t$, contradiction! Thus $M=0$. 
    
\end{proof}

We record the following remark in regards to the hypothesis of $I$ containing a non-zero-divisor in Theorem \ref{big}. 

\begin{rem} Let $I$ be an ideal of $R$. Assume $\pd_R(I)<\infty$. Then $I$ contains a non-zero-divisor of $R$ if and only if $I_{\m}\neq 0$ for all $\m \in \Max \spec(R)$. Indeed, let $I$ contains a non-zero-divisor of $R$, say $c$. Then $\frac{c}{1}\in I_{\m}$ for all $\m \in \Max \spec(R)$. Now $\frac{c}{1}\neq 0$ in $I_{\m}$ for all $\m \in \Max \spec(R)$. Indeed, if $\frac{c}{1}=0$ in $I_{\m}$ for some $\m \in \Max \spec(R)$, then $\frac{c}{1}=\frac{0}{1}$. This implies $cs=0$ for some $s\in R\setminus\m$, which means $s=0$, which is a contradiction as $s\in R\setminus \m$. Thus $\frac{c}{1}\neq 0$ in $I_{\m}$ for all $\m \in \Max \spec(R)$, so $I_{\m}\neq 0$ for all $\m \in \Max \spec(R)$. Conversely, let $I_{\m}\neq 0$ for all $\m \in \Max \spec(R)$. Now $\pd_R(I)<\infty$ implies $\pd_{R_{\m}}(I_{\m})<\infty$ for all $\m \in \Max \spec(R)$, which implies $\pd_{R_{\m}}(R_{\m}/I_{\m})<\infty$ for all $\m \in \Max \spec(R)$. Now $\ann_{R_{\m}}(R_{\m}/I_{\m})=I_{\m}\neq 0$ for all $\m \in \Max \spec(R)$, so by \cite[Proposition 6.2]{Ab} 
we get that $I_{\m}$ contains a non-zero-divisor of $R_{\m}$ for all $\m \in \Max \spec(R)$. Hence by \ref{nw} we get that, $\ann_{R_{\m}}(I_{\m})=0$ for all $\m \in \Max \spec(R)$. Now $(\ann_R(I))_{\m}=\ann_{R_{\m}}(I_{\m})$ for all $\m \in \Max \spec(R)$. Thus we have $(\ann_R(I))_{\m}=0$ for all $\m \in \Max \spec(R)$, so $\supp(\ann_R(I))=\emptyset$. Hence $\ann_R(I)=0$, so \ref{nw} implies that $I$ contains a non-zero-divisor of $R$.
\end{rem}

Next, we aim towards giving a consequence of Theorem \ref{4} for trace ideals of big Cohen--Macaulay modules.  
The following lemma is probably well-known, but we record a proof for the lack of an exact reference. In the following, we use the notion of big Cohen-Macaulay modules as in \cite[3.1]{bcm}. For basics of complexes and derived functors in the context of commutative algebra, we refer the reader to \cite[Appendix]{ch}. 

\begin{lem}\label{bigcm} Let $(R,\m)$ be a complete local Cohen--Macaulay ring. Let $M$ be a big Cohen--Macaulay $R$-module. Then $\Ext_R^{>0}(M,D)=0$ and $\Hom_R(M,D)\neq 0$, where $D$ is a dualizing module of $R$.   
\end{lem}  

\begin{proof} Let $d:=\dim R$, and $E$ be the injective hull of  $R/\m$.  We know  $\mathbf R\Gamma_{\m}(D)\cong \Sigma^{-d} E$ (\cite[3.4]{bcm}). By \cite[Example 7.3.2(d), Theorem 9.1.3(b), Theorem 7.5.12]{ss} we get 

$$\mathbf L \Lambda^{\m} (D)\cong \mathbf L \Lambda^{\m} (\mathbf R\Gamma_{\m}(D)) \cong \mathbf R \Hom_R(C, \Sigma^{-d} E)$$

% \old{$$\mathbf L \Lambda^{\m} (R)\cong \mathbf L \Lambda^{\m} (\mathbf R\Gamma_{\m}(R)) \cong \mathbf R \Hom_R(C, \Sigma^{-d} E)$$}

where $C$ is the \v{C}ech complex on a minimal generating set of $\m$. Since $R$ is complete, hence $D$ is $\m$-adically complete by \cite[Theorem 3.3.5(c)]{bh}. Thus  $D\cong \mathbf L \Lambda^{\m} (D)$ by \cite[Corollary 2.5.8, Corollary 7.5.13(a)]{ss}. Now we get 

$$\mathbf R\Hom_R(M,D)\cong \mathbf R\Hom_R(M,\mathbf R \Hom_R(C, \Sigma^{-d} E))\cong \mathbf R\Hom_R(M\otimes_R^{\mathbf L}C, \Sigma^{-d}E)\cong \Sigma^{-d}\mathbf R\Hom_R(\mathbf R\Gamma_{\m}(M), E)$$

% \old{$$\mathbf R\Hom_R(M,R)\cong \mathbf R\Hom_R(M,\mathbf R \Hom_R(C, \Sigma^{-d} E))\cong \mathbf R\Hom_R(M\otimes_R^{\mathbf L}C, \Sigma^{-d}E)\cong \Sigma^{-d}\mathbf R\Hom_R(\mathbf R\Gamma_{\m}(M), E)$$}

where in the second isomorphism, we have used \cite[A.4.21]{ch}, and in the third isomorphism, we have used \cite[A.2.1.1]{ch} and the last equality of \cite[Section 2.1]{chiy}. By \cite[1.1, 3.1(3)]{bcm} we get, $\mathbf R\Gamma_{\m} (M)\cong \Sigma^{-d} \H^d_{\m} (M)$. Hence using \cite[A.2.1.3]{ch} we get, 

$$\mathbf R\Hom_R(M,D)\cong  \Sigma^{-d}\Sigma^d\mathbf R\Hom_R(\H^d_{\m} (M), E)\cong \mathbf R\Hom_R(\H^d_{\m} (M), E)\cong \Hom_R(\H^d_{\m} (M), E)$$

% \old{$$\mathbf R\Hom_R(M,R)\cong  \Sigma^{-d}\Sigma^d\mathbf R\Hom_R(\H^d_{\m} (M), E)\cong \mathbf R\Hom_R(\H^d_{\m} (M), E)\cong \Hom_R(\H^d_{\m} (M), E)$$}

where the last isomorphism follows because $E$ is an injective module. Now comparing homologies we get, $\Ext_R^{>0}(M,D)=0$ and $\Hom_R(M,D)\cong \Hom_R(\H^d_{\m} (M), E)$. By \cite[3.1(3)]{bcm} we know $\H^d_{\m} (M) \neq 0$. Hence from \cite[1.4.8]{ss} we get that, $\Hom_R(\H^d_{\m} (M), E)$ is also non-zero. Thus $\Hom_R(M,D)\neq 0$.  
\end{proof}

The hypothesis of $R$ being complete cannot be dropped from Lemma \ref{bigcm} as is shown by the following remark. 

\begin{rem} If $(R,\m)$ is a Cohen--Macaulay local ring, admitting a dualizing module $D$, such that $\Ext^1_R(M,D)=0$ for every balanced big Cohen--Macaulay module $M$, then $R$ is complete. Indeed, we first notice by \cite[Proposition 2.4, Remark 4.4]{holm} that $F\oplus R$ is balanced big Cohen--Macaulay for every flat $R$-module $F$. In particular, if $x_1,\cdots,x_d$ is an $\m$-primary ideal, then the direct sum of $R$ and all the $R_{x_i}$ s is a balanced big Cohen--Macaulay module, hence the required Ext vanishing  forces $D$ to be complete by \cite[Theorem 1.1]{sch}. Hence $D\cong D\otimes_R \widehat R$, which implies $R\cong \Hom_R(D,D)\cong \Hom_R(D, D\otimes_R \widehat R)\cong \Hom_R(\Hom_R(D,D),\widehat R)\cong \Hom_R(R,\widehat R)\cong \widehat R$, i.e., $R$ is complete.   
\end{rem}

% \begin{proof} Since $\m R\neq R$, hence $\m (K\oplus R)\neq K\oplus R$. Also, $K\oplus R$ is torsion-free and $R$ is a domain, hence every $R$-regular element is $K\oplus R$-regular. Since $\dim R=1$, this shows that $K\oplus R$ is balanced big Cohen--Macaulay. By hypothesis, $\m=pR$  for some prime element $p\in \m$. Since $K=\cup_{i\geq 1} p^{-i} R$, and $\Ext_R^1(p^{-i}R,R)\cong \Ext^1_R(R,R)=0$. From \cite[3.5.10]{wei}, we then get that $\Ext^1_R(K,R)\cong \varprojlim^1 \Hom_R(p^{-i}R,R)$. 
    
% \end{proof}

We need one more lemma before proving our main result (Theorem \ref{bigmac}) on trace ideals of big Cohen--Macaulay modules. 

\begin{lem}\label{injorth} Let $(R,\m)$ be a complete local Cohen--Macaulay ring. If $N$ is a finitely generated $R$-module of finite injective dimension and $M$ is a big Cohen--Macaulay $R$-module, then $\Ext_R^{>0}(M,N)=0$. 
\end{lem}

\begin{proof} Let $\omega$ be a canonical module of $R$. By \cite[3.3.28(b)]{bh}, there exists  a finite $\omega$-resolution of $N$ $0\to \omega^{\oplus b_n}\to \cdots \to \omega^{\oplus b_0}\to N\to 0$ for some non-negative integers $b_0,\cdots,b_n$ and $n$. We prove the claim by induction on $n$, the length of $\omega$-resolution. If $n=0$ we are done by Lemma \ref{bigcm}.   Now assume $n\geq 1$. We then have exact sequences $0\to \omega^{\oplus b_n}\to \cdots \to \omega^{\oplus b_1}\to N'\to 0$ and $0\to N'\to \omega^{\oplus b_0}\to N\to 0$. By induction hypothesis, $\Ext^{>0}_R(M,N')=0$.  Then by the long exact sequence of Ext after applying $\Hom_R(M,-)$ to $0\to N'\to \omega^{\oplus b_0}\to N\to 0$, and using Lemma \ref{bigcm}, it follows that $\Ext^{>0}_R(M,N)=0$. This finishes the inductive step, hence the proof.    
\end{proof}

\begin{thm}\label{bigmac} Let $(R,\m)$ be a complete local Cohen--Macaulay  ring. Let $M$ be a big Cohen--Macaulay $R$-module and $I\subseteq \m$ be an ideal of $R$ with $\pd_R I<\infty$. If $\omega$ is a dualizing module of $R$, then $\tau_{\omega}(M)\nsubseteq I\omega$. In particular, if $R$ is Gorenstein, then $\tr_R(M)\nsubseteq I$. 
\end{thm}

\begin{proof} As $\pd_R (R/I)<\infty$ and $\omega$ is maximal Cohen--Macaulay, so $\Tor^R_{>0}(\omega, R/I)=0$ (see \cite[Definition 5.3.1, Theorem 5.3.10]{ch}). Since $\syz_R I$ has finite projective dimension, so $\omega\otimes_R \syz_R I$ has finite injective dimension (by \cite[Exercise 9.6.5]{bh}). Also, $\Ext^{>0}_R(M, \omega \otimes_R \syz_R I)=0$ by Lemma \ref{injorth}. Thus with $N=\omega$ in Theorem \ref{4}(1), and from Lemma \ref{bigcm} we get that, $\tau_{\omega}(M)\nsubseteq I\omega$.  
\end{proof}

For our next consequence of Theorem \ref{4}, we first recall that, a finitely generated $R$-module $M$ is said to satisfy Serre's condition $(\widetilde S_n)$ for some non-negative integer $n$ if $\depth_{R_{\p}} M_{\p} \ge \inf\{n, \depth R_{\p}\}$ for all $\p\in \spec(R)$.  We first record a remark about this Serre's condition $(\widetilde S_n)$.

\begin{rem}\label{stable} Let $M,N$ be two finitely generated $R$-modules such that $M\oplus F\cong N\oplus G$ for some finitely generated free $R$-modules $F, G$. If $M$ satisfies $(\widetilde S_n)$ for some non-negative integer $n$, then so does $N$. Indeed, let $\p\in\spec R$. Then we have $M_{\p}\oplus F_{\p}\cong N_{\p}\oplus G_{\p}$, which implies $$\inf\{\depth_{R_{\p}}M_{\p},\depth R_{\p}\}=\depth_{R_{\p}}(M_{\p}\oplus F_{\p})= \depth_{R_{\p}}(N_{\p}\oplus G_{\p})=\inf\{\depth_{R_{\p}}N_{\p},\depth R_{\p}\}$$ Now we have \begin{align*}
\depth_{R_{\p}}N_{\p}&\ge \inf\{\depth_{R_{\p}}N_{\p},\depth R_{\p}\}\\
&= \inf\{\depth_{R_{\p}}M_{\p},\depth R_{\p}\}\\
&\ge \inf\{\inf\{n, \depth R_{\p}\},\depth R_{\p}\}\text{ [Since }M\text{ satisfies } (\widetilde S_n)]\\
&=\inf\{n, \depth R_{\p}\}
\end{align*}
Hence $N$ satisfies $(\widetilde S_n)$ as well.
\end{rem}

Now the $h=0$ case of Proposition \ref{np} recovers \cite[Corollary 2.2(b)]{hr} and our proof of Proposition \ref{np} does not depend on \cite[Corollary 2.2(b)]{hr}.      

\begin{prop}\label{np} Let $I$ be an ideal of $R$ such that $ I\subseteq J(R)$, and $M$ be a finitely generated $R$-module. Let $\pd_R(I)\leq h<\infty$. 
Assume $\Tr_R M$ satisfies $(\widetilde S_h)$ and $\tr_R(M)\subseteq I$. Then $M^*=0$, i.e., $\tr_R(M)=0$, i.e., $M$ is torsion. If, moreover, $R$ is local, and  $h\geq 1\geq \depth R$, then $M=0$. 
%If, moreover, $\depth R\leq 1$, then $M=0$. 
\end{prop} 

\begin{proof} We first notice that both the hypothesis and the conclusion of the first part localizes with respect to maximal ideals. Indeed the only non-trivial part to see here is that: (1)$\Tr_{R_{\m}} M_{\m}$ satisfies $(\widetilde S_h)$ for every $\m \in \text{Max}\spec(R)$ by \cite[Remarks (5), pp. 5789]{mas} and Remark \ref{stable}; and (2) the condition $\tr_R(M)\subseteq I$ localizes, which follows from \cite[Proposition 2.8(viii)]{lindo}. Hence, without loss of generality, we may assume $R$ is local. By \cite[A.5(i), Remark A.6]{dck} we have $0=\Tor^R_{>0}(\Tr M, I)$. Applying Theorem \ref{4}(2) with $N=R$, we get $M^*=0$. Finally, $M^*=0$ if and only if $M$ is torsion by \cite[Proposition 1.2.3]{bh}.

On the other hand, remember that $M^*=0$ implies $\pd_R(\Tr M)\leq 1$.  

 Now for the moreover part, assume $R$ is local with $\depth R\leq 1\leq h$. If $\depth R=0$, then by Auslander-Buchsbaum formula, $\Tr M$ is free, hence $M$ is free. If $\depth R=1$, then $\Tr M$ satisfying $(\widetilde S_h)$ yields $\depth_R (\Tr M)\geq \inf\{h,1\}\geq 1$ as $h\geq 1$. Then $1=\depth R=\depth(\Tr M)+\pd(\Tr M)\geq 1+\pd(\Tr M)$ yields $\Tr M$ is free, hence $M$ is free. Now $M^*=0$ implies $M=0$. 
\end{proof}

Now we can state and prove an obstruction result about containment of trace of a finitely generated module inside ideals of finite projective dimension, when the module satisfies certain Ext vanishing conditions. The $h=0$ case of Theorem \ref{hah}(1) recovers \cite[Corollary 2.2(b)]{hr}. Note that, when $M$ is finitely generated, Theorem \ref{hah}(1) is more general than Theorem \ref{big}.

\begin{thm}\label{hah} Let $I$ be an ideal of $R$ such that $ I\subseteq J(R)$, and $M$ be a finitely generated $R$-module. Let $\pd_R(I)\leq h<\infty$ and $\tr_R(M)\subseteq I$. Then the following hold true. 
\begin{enumerate}[\rm(1)]
    \item If $\Ext_R^{1\leq i\leq h}(M,R)=0$, then $M^*=0$, i.e., $\tr_R(M)=0$.
    \item If $(R,\m)$ is local, $\depth R=t$, and $\Ext_R^{1\leq i\leq t}(M,R)=0$, then $M=0$.
\end{enumerate} 
\end{thm}  

\begin{proof} (1) Since $\Ext_R^{1\leq i\leq h}(\Tr \Tr M,R)=0$, so $\Tr M$  satisfies $(\widetilde S_h)$ (see \cite[Definition 7, Proposition 11]{mas}). Now we are done by Proposition \ref{np}. 

(2)  Since $R$ is local, so by Auslander-Buchsbaum formula we have $\pd_R(I)\leq t$. Hence by (1) we get $M^*=0$. Thus $\Ext_R^{0\leq i\leq t}(M,R)=0$. 
If $M\neq 0$, then $\m \in \supp_R(M)$. Hence  \cite[Proposition 1.2.10(a), (e)]{bh} would imply that $\inf \{n: \Ext_R^n(M,R)\neq 0\}\leq \depth R=t$, contradiction! Thus $M=0$. 
\end{proof}

Now we present a variation of Theorem \ref{bigmac}, where we do not assume the ring is complete at the cost of working with only finitely generated modules.

\begin{cor}\label{6} Let $(R,\m)$ be a local Cohen--Macaulay ring. Let $I\subseteq \m$ be an ideal of $R$ of finite projective dimension. Let $M$ be a maximal Cohen--Macaulay $R$-module. Assume $R$ admits a canonical module $\omega$ such that $\tau_{\omega}(M)\subseteq I\omega$. Then $M=0$. 
\end{cor}  

\begin{proof} As $\pd_R (R/I)<\infty$ and $\omega$ is maximal Cohen--Macaulay, so $\Tor^R_{>0}(\omega, R/I)=0$ (see \cite[Definition 5.3.1, Theorem 5.3.10]{ch}). Since $\syz_R I$ has finite projective dimension, so $\omega\otimes_R \syz_R I$ has finite injective dimension (by \cite[Exercise 9.6.5]{bh}). Since $M$ is maximal Cohen--Macaulay, so $\Ext^{>0}_R(M, \omega \otimes_R \syz_R I)=0$ by \cite[Exercise 3.1.24]{bh}. Thus with $N=\omega$ in Theorem \ref{4}(1), we get $\Hom_R(M,\omega)=0$. As $M$ is maximal Cohen--Macaulay, so we get $M\cong \Hom_R(\Hom_R(M,\omega),\omega)=0$. 
\end{proof}   

Now we record a series of preliminary observations needed for proving our obstruction result regarding containment in ideals of finite injective dimension.

\begin{chunk}\label{nw}Note that, for any ideal $J$ of a ring $R$, $J$ contains a non-zero-divisor of $R$ if and only if $\ann_R(J)=0$. Indeed, if $J$ contains a non-zero-divisor, then it is clear that $\ann_R(J)=0$. Conversely, assume $\ann_R(J)=0$. If $J$ consisted of zero-divisors, then $J\subseteq \p=\ann_R(x)$ for some $\p\in \Ass(R)$ and $0\neq x \in R$. But then, $0\neq x \in \ann_R(J)$.
\end{chunk}

Since an $R$-module $M$ is faithful if and only if the natural homothety $R\xrightarrow{r\to (x\mapsto rx)} \End_R(M)$ is injective, so we immediately get the following:  

\begin{chunk}\label{fc} Let $(R,\m)$ be local and $M$ be an $R$-module. Then $\ann_R(M)=0$ iff $\ann_{\widehat R}(\widehat M)=0$.
\end{chunk}   

Using \ref{fc} and \ref{nw}, we immediately get: 

\begin{chunk}\label{nzcomple} Let $I$ be an ideal of a local ring $(R,\m)$. Then $I$ contains a non-zero-divisor of $R$ if and only if $I\widehat R$ contains a non-zero-divisor of $\widehat R$.
\end{chunk}

\begin{chunk}\label{37} Let $0\neq M$ be a finitely generated module over a local ring $(R,\m)$ such that $\id_R M<\infty$. If $\ann_R(M)\neq 0$, then it contains a non-zero-divisor. 

\begin{proof} By Bass's theorem, $R$ is Cohen--Macaulay. By \ref{fc} and \ref{nzcomple}, we may assume $R$ is complete. Hence we may assume $R$ admits a canonical module, say $\omega$. Then $\pd_R \Hom_R(\omega, M)<\infty$ and $M\cong \omega \otimes_R \Hom_R(\omega, M)$ (by \cite[Exercise 9.6.5]{bh}).   Hence $0\neq \ann_R(M)=  \ann_R \Hom_R(\omega, M)$. Since $\ann_R \Hom_R(\omega, M)$ contains a non-zero-divisor by \cite[Proposition 6.2]{Ab}, hence so does  $\ann_R(M)$.

\end{proof}
\end{chunk}

\begin{lem}\label{38} Let $0\neq I$ be an ideal of a local ring $(R,\m)$ such that $\id_R I<\infty$.  Then $I$ contains a non-zero-divisor of $R$ and $R$ is generically Gorenstein. 
\end{lem} 

\begin{proof} If $I$ consists of zero-divisors, then $\ann_R(I)\neq 0$ by \ref{nw}. This would imply that $\ann_R(I)$ contains a non-zero-divisor by \ref{37}, i.e., $rI=0$ for some non-zero-divisor $r\in R$. But then, $I=0$, contradiction! Thus $I$ must contain a non-zero-divisor. Hence $I\nsubseteq \p$ for every $\p\in \Ass(R)$. Thus $I_{\p}=R_{\p}$ for every $\p\in \Ass(R)$. Hence for every $\p\in \Ass(R)$, $\id_{R_{\p}} R_{\p}=\id_{R_{\p}} I_{\p}<\infty$, so $R_{\p}$ is Gorenstein. Hence $R$ is generically Gorenstein.
\end{proof}

% \begin{chunk} Let $(R,\m)$ be a local Cohen--Macaulay ring of positive dimension such that $\widehat R$ is generically Gorenstein. If $0\neq I$ is an ideal of $R$ of finite injective dimension, then $I$ contains a non-zero-divisor.  

% \begin{proof}
    
% \end{proof}
% \end{chunk} 

Lemma \ref{trfrc} generalizes \cite[Proposition 2.4(2)]{tk} and follows from \cite[Proposition 2.4(1)]{tk}. 

\begin{lem}\label{trfrc} Let $M,N$ be finitely generated $R$-submodules of $Q(R)$. If $N$ contains a non-zero-divisor of $R$, then $\tau_M(N)=(M:N)N$. 
\end{lem}  

% \begin{lem} Let $(R,\m)$ be a local Cohen--Macaulay ring.  If there exists  a  finitely generated $R$-submodule $0\neq J$ of $Q(R)$ of finite injective dimension, then $R$ is generically Gorenstein. 
% \end{lem}

% \begin{proof} 
    
% \end{proof}

\begin{lem}\label{qr} Let $(R,\m)$ be a local Cohen--Macaulay  ring admitting a canonical module.  Let $0\neq J$ be a  finitely generated $R$-submodule of $Q(R)$ such that $\id_R J<\infty$. Then $R$ is generically Gorenstein, and for every canonical module $\omega\subseteq Q(R)$, it holds that $(J:\omega)$ has finite projective dimension, and $J=(J:\omega)\omega$
\end{lem}

\begin{proof} Since $I\cong J$ for some ideal $I$ of $R$, hence $R$ is generically Gorenstein by Lemma \ref{38}. In particular, there exists a canonical module $\omega\subseteq Q(R)$ (by \cite[Proposition 3.3.18]{bh}). By \cite[Exercise 3.3.28(b)]{bh} we get $\tau_J(\omega)=J$. Since  $\omega$ contains a non-zero-divisor by \cite[Proposition 3.3.18]{bh}, hence from Lemma \ref{trfrc} we get $J=(J:\omega)\omega$. Moreover, $(J:\omega)\cong \Hom_R(\omega, J)$ (by \cite[Proposition 2.4(1)]{tk}) has finite projective dimension by \cite[Exercise 9.6.5]{bh}.  

% We have $J\cong \Hom_R(\omega,J)\otimes_R \omega$ (\cite[9.6.5]{bh}). Since $\omega$ contains a non-zero-divisor, hence $(J:\omega)\cong \Hom_R(\omega, J)$ (\cite[Proposition 2.4(1)]{tk}), so it has finite projective dimension (\cite[9.6.5]{bh}). Since $(J:\omega)$ is a finitely generated submodule of $Q(R)$, so $(J:\omega)\cong I$ for some ideal $I\subseteq R$. Thus, $J\cong I\otimes_R \omega$. Since $\pd_R(R/I)<\infty$, so $\Tor^R_1(R/I,\omega)=0$, so $I\omega\cong I\otimes_R \omega \cong J $. Since $J\neq 0$, so $I\neq 0$. Now $\pd_R(I)<\infty$ implies  $I$ contains a non-zero-divisor (\cite[Corollary 6.3]{Ab}), so $I\omega$, and hence also $J$ contains a non-zero-divisor. Then, $aI\omega=bJ$ for some non-zero-divisors $a,b\in Q(R)$ such that $ab=1$ (\cite[Proposition 2.4(1)]{tk}). Set $y=a/b\in Q(R)$. Then, $yI\subseteq (J:\omega)$. So, $J=(yI)\omega\subseteq (J:\omega)\omega$. The reverse inclusion $(J:\omega)\omega\subseteq J$ is obvious.  

\end{proof}

Now finally, we can state and prove our obstruction result regarding containment in ideals of finite injective dimension.

\begin{cor}\label{injmain}  Let $(R,\m)$ be a  local Cohen--Macaulay  ring  admitting a canonical module. Let $M$ be a maximal Cohen--Macaulay $R$-module. Let $0\neq J (\subseteq \m)$ be an  ideal of finite injective dimension. Then $R$ is generically Gorenstein, there exists a  canonical module $\omega$ satisfying $R\subseteq \omega\subseteq Q(R)$, and for some such canonical module, if $\tau_{\omega}(M)\subseteq J$ holds, then we must have $M=0$. 
\end{cor}  

\begin{proof} $R$ is generically Gorenstein by  Lemma \ref{38}. Hence $R$ admits a canonical ideal $ C\subseteq R$ containing a non-zero-divisor (\cite[Proposition 3.3.18]{bh}), say $a\in C$. Then $\dfrac 1 a C$ is a canonical module containing $R$ and contained in $Q(R)$. We know $J=(J:\omega)\omega$ from Lemma \ref{qr}. Since $1\in \omega$, so $(J:\omega)\subseteq J\subseteq \m$. So, $I:=(J:\omega)$ is an ideal, contained in $\m$, of finite projective dimension (by Lemma \ref{qr}).  Now if $\tau_{\omega}(M)\subseteq J$ holds, then we are done by Corollary \ref{6}.  
\end{proof}

\begin{rem}\label{3.22} If $R$ is complete in Corollary \ref{injmain}, then we can use big Cohen--Macaulay modules $M$ instead of just maximal Cohen--Macaulay modules and conclude $\tau_{\omega}(M)\nsubseteq J$, by using Theorem \ref{bigmac} in the proof in place of Corollary \ref{6}.  
\end{rem}

\section{connections with ideals of entries in free resolutions}\label{trid} 

In this section, all modules are assumed to be finitely generated.

For a free resolution $(F_i,\partial_i)_{i=0}^\infty$ of a module $M$, we put $\partial_0=0$ by convention. $\I_1(\partial_j)$ will stand for the ideal generated by the entries of a matrix representation of the map $\partial_j$. This definition is independent of the choice of the matrix representation. When $R$ is local, $\syz^i_R M$ will stand for $\text{Im}(\partial_i)$, where $(F_i,\partial_i)_{i=0}^\infty$ is a minimal free resolution of $M$. 

In this section, we provide some comparison results between trace ideals and ideals of entries of maps in a free resolution of a module. These in turn provide obstruction results for containment of such ideals of entries in ideals of finite projective dimension. Our first such result is about modules, which are regular fractional ideals.

\begin{prop}\label{trentry} Let $I$ be a non-free $R$-submodule of $Q(R)$. Assume $I$ contains a non-zerodivisor of $R$. If $(F_i,\partial_i)$ is a free resolution of $I$,  then $\tr_R(I)\subseteq \I_1(\partial_1)$.  If  $y\in J(R)$ is a non-zero-divisor, then $\I_1(\partial_1)\nsubseteq yR$. 
\if0
Then the following hold:
\begin{enumerate}[\rm(1)]
\item If $I$ can be generated by $n$ elements, then $I$ can be generated by $n$-many non-zerodivisors of $R$.  

    \item If $(F_i,\partial_i)$ is a free resolution of $I$,  then $\tr_R(I)\subseteq \I_1(\partial_1)$. 
\end{enumerate}
\fi

\end{prop}

\begin{proof} Let $\pi: F_0 \to \frac{F_0}{\text{Im}(\partial_1)}\cong I$ be the natural projection. Since $I$ is non-principal, so $F_0\cong R^{\oplus n}$ for some $n\ge 2$. Let $e_1,\dots,e_n$ be a basis of $F_0$. Set $\pi(e_i)=f_i \in I$ for all $i=1,\cdots,n$. So, in particular,  $f_1,\dots,f_n$ generate $I$. Let $x\in (R:_{Q(R)}I)$. Fix an $i\in\{1,\cdots,n\}$. Then $xf_i\in (R:_{Q(R)}I)I\subseteq R$. Pick $j\in\{1,\cdots,n\}$ such that $j\neq i$ (this can be done as $n\ge 2$). Consider $(xf_i) e_j-(xf_j)e_i\in F_0$. Now $\pi\left((xf_i) e_j-(xf_j)e_i \right)=(xf_i)\pi(e_j)-(xf_j)\pi(e_i)=xf_if_j-xf_jf_i=0$. So, $(xf_i) e_j-(xf_j)e_i\in \ker \pi=\text{Im}(\partial_1)$. Thus $xf_i\in \I_1(\partial_1)$. Since $i$ was arbitrary and $f_1,\dots,f_n$ generate $I$, so $(R:_{Q(R)}I)I\subseteq \I_1(\partial_1)$. Thus we conclude $\tr_R(I)\subseteq \I_1(\partial_1)$ by \cite[Proposition 2.4(2)]{tk}.  If  $y\in J(R)$ is a non-zero-divisor and $\I_1(\partial_1)\subseteq yR$, then $\tr_R(I)\subseteq y R$. So, $I^*=0$ by Proposition \ref{np} (applied to $h=0$). Since $I$ is torsion-free, hence $I=0$, contradicting $I$ contains a non-zero-divisor. 

% \if0
% (1) Let $\{\mathfrak p_1,\cdots,\mathfrak p_r\}$ be the set of associated primes of $R$. Also, let $I=(y_1,\cdots,y_n)$. Then 
% $I=y_1R+J$, where $J=(y_2,\cdots,y_n)$. Since $I$ contains a non-zerodivisor of $R$, so $I\not\subset\mathfrak p_i$ for each $i$. Hence by E. Davis' prime avoidance, there exists some $x_1\in J$ such that $x_1=y_1+\sum_{j=2}^nr_jy_j$ for some $r_j\in R$ and $x_1\not \in \mathfrak p_i$ for each $i$, so $x_1$ is a non-zerodivisor of $R$. Also, $x_1=y_1+\sum_{j=2}^nr_jy_j$ implies that $I=(y_1,y_2,\cdots,y_n)=(x_1,y_2,\cdots,y_n)$. Now note that, $I=(y_2,x_1,\cdots,y_n)$. Then by following the same steps we can replace $y_2$ with a non-zerodivisor $x_2$ of $R$, which gives us $I=(x_1,x_2,y_3,\cdots,y_n)$. Thus proceeding in this way, we will get $I=(x_1,x_2,\cdots,x_n)$, where each $x_j$ is a non-zerodivisor. Hence $I$ can be generated by $n$-many non-zerodivisors of $R$. 
% \\
% (2) 
% \fi  
\end{proof}

Our next result connects trace of syzygy and ideal of entries of some map in a minimal free resolution of a module under exactly one Ext vanishing.

We may frequently use \ref{34} without further reference.  

\begin{chunk}\label{34} Let $M$ be an $R$-module and $(F_i,\partial_i)_{i=0}^\infty$ be a free resolution of $M$, then We have isomorphisms $\Ext^i_R(\Im(\partial_j), -)\cong \Ext^{i+j}(M,-)$ for all $i\geq 1, j\geq 1$.
\end{chunk}

%In the following, $\Tr(-)$ will denote Auslander-transpose    

\begin{lem}\label{43}  Let $M$ be an $R$-module with a free resolution $(F_i,\partial_i)_{i=0}^\infty$. Then the following holds:

\begin{enumerate}[\rm(1)]  
\item Let $j\ge 1$ be an integer such that $\Ext^j_R(M,R)=0$. Then $\I_1(\partial_j)=\tr_R(\im(\partial_j))$. 
    
\item Let $R$ be local and $j\ge 1$ be an integer such that $\Ext^j_R(M,R)=0$. If $(F_i,\partial_i)_{i=0}^\infty$ is a minimal free resolution, then $\I_1(\partial_j)=\tr_R(\syz^j_R M)$. 
\end{enumerate}
\end{lem}

\begin{proof} (1)  For each $n\ge 1$, we can factor the map $F_n\xrightarrow{\partial_n}F_{n-1}$ as follows:

$$\begin{tikzcd}
F_n \arrow[rd, "p_n"'] \arrow[rr, "\partial_{n}"] &                                    & F_{n-1} \\                                                    & \im(\partial_n) \arrow[ru, "i_n"'] &        
\end{tikzcd}$$

where $i_n$ is just inclusion. Since $(F_i,\partial_i)_{i=0}^\infty$ is acyclic, so $\coker(\partial_n)\cong \Im(\partial_{n-1})$ for all $n\ge 2$, and also $\coker(\partial_0)\cong M$. Hence we notice that,  $\Ext^i_R(\coker(\partial_n),-)\cong \Ext^{i+n-1}_R(M,-)$ for all $i\ge 1, n\ge 1$. 

Now we have exact sequences $\alpha:\text{ } F_{j+1}\xrightarrow{\partial_{j+1}} F_j \xrightarrow{p_j} \im(\partial_j)\to 0$ and $\beta:\text{ }0\to \im(\partial_j)\xrightarrow{i_j}F_{j-1}\to \coker(\partial_{j})\to 0$. Dualizing the sequence $\alpha$ by $R$ we get the exact sequence $$0\to \im(\partial_j)^*\xrightarrow{p_j^*}F_j^*\xrightarrow{\partial_{j+1}^*}F_{j+1}^*\to \Tr(\im(\partial_j))\to 0 $$ Since $\Ext^1_R(\coker(\partial_{j}),R)\cong\Ext^j_R(M,R)=0$, so dualizing the sequence $\beta$ by $R$ we get the exact sequence  $$0\to \coker(\partial_{j})^*\to F_{j-1}^*\xrightarrow{i_j^*}\im(\partial_j)^*\to 0$$  Hence we get an exact sequence $$F_{j-1}^*\xrightarrow{p_j^*\circ i_j^*}F_j^*\xrightarrow{\partial_{j+1}^*}F_{j+1}^*\to \Tr(\im(\partial_j))\to 0$$    Now  $p_j^*\circ i_j^*=(i_j\circ p_j)^*=\partial_j^*$.   In view of \cite[Proposition 3.1]{he} we then get $\tr_R(\im(\partial_j))=\I_1(p_j^*\circ i_j^*)=\I_1(\partial_j^*)=\I_1(\partial_j)$.  

(2) This follows from (1) by taking a minimal free resolution $(F_i,\partial_i)_{i=0}^\infty$ and recalling that $\syz^j_R M=\im(\partial_j)$ by definition. 
\end{proof}

 Our next result shows that if the ideal of entries of certain map in a free resolution of a module is contained in an ideal of finite projective dimension, then under some additional Ext vanishing hypotheses, the module must have finite projective dimension.

\begin{cor}\label{44}   Let $I$ be an ideal of $R$ such that $ I\subseteq J(R)$ and $\pd_R (I)\leq h<\infty$. Let $M$ be an $R$-module and $j\geq 1$ be an integer such that $\Ext_R^{j\leq i \leq j+h}(M, R)=0$. If $(F_i,\partial_i)_{i=0}^\infty$ is a free resolution of $M$, and $\I_1(\partial_j)\subseteq I$, then $\pd_R M<j$. 
\end{cor}    

\begin{proof} Since $\Ext_R^{j\leq i \leq j+h}(M, R)=0$, so $\I_1(\partial_j)=\tr_R(\Im(\partial_j))$ by Lemma \ref{43}(1). If  $\I_1(\partial_j)\subseteq I$, then $\tr_R(\Im(\partial_j))\subseteq I$. Since $\Ext_R^{1\leq i\leq h}(\Im(\partial_j), R)=0$ (see \ref{34}), so by Theorem \ref{hah}(1) we get, $\Im(\partial_j)$ is torsion. However, $j\geq 1$ implies $\Im(\partial_j)$ is torsion-free. Hence $\Im(\partial_j)=0$, and so $\pd_R M<j$.   
\end{proof}

\section{order of trace ideals }\label{ortr}   

%Using \cite[Remark 1.5, Theorem 3.5(ii)]{kl} and \cite[Corollary 1.2(b)]{hr}, we get the following:    

In this section, all modules are assumed to be finitely generated.

In this section, we apply the results from Section \ref{secmain} to establish some upper bounds on $\m$-adic order of certain trace ideals of local Cohen--Macaulay ring $(R,\m)$. In the following, $e(R)$ will denote the Hilbert-Samuel multiplicity of $R$. 

By Krull-intersection theorem, for any non-zero ideal $I$ of a local ring $(R,\m)$, the following quantity is finite $$\ord(I):=\sup \{n\in\mathbb{Z}_{\ge 0}\text{ }|\text{ } I\subseteq \m^n\}$$  

Note that, $\ord(I)=0$ if and only if $I=R$. We also denote by  $\ell\ell(R)$ the generalized Loewy length of $R$ defined as $$\ell\ell(R):=\inf \{i: \m^i\subseteq (x_1,\cdots,x_d) \text{ for some system of parameters } x_1,\cdots,x_d \text{ of } R \}$$ (see \cite[Remark 1.5]{kl}). It is clear that if $(R,\m)$ is Artinian local, then $\m^{\ell\ell(R)}=0$. 

\begin{prop}\label{her} Let $(R,\m)$ be a local Cohen--Macaulay ring of dimension $1$ with infinite residue field. Let $I\subseteq \m$ be a non-zero ideal of $R$ such that $I=\tr_R(I)$. Then $1\le \ord(I)\le \ell\ell(R)-1\le e(R)-\mu(\m)+1$. In particular, equality holds throughout when $R$ has minimal multiplicity. 
\end{prop}

\begin{proof}  
First note that, $\ell\ell(R)-1\le e(R)-\mu(\m)+1$ follows from \cite[Remark 1.5, Theorem 3.5(ii)]{kl}. Now we will prove $\ord(I)\le \ell\ell(R)-1$ by contradiction. So, if possible let, $\ord(I)\ge \ell\ell(R)$. This implies $I\subseteq \m^{\ell\ell(R)}$. Now by the definition of $\ell\ell(R)$ we have, $\m^{\ell\ell(R)}\subseteq (x)$, where $x\in\m$ is a non-zero divisor. Thus $\tr_R(I)=I\subseteq xR$, so $I^*=0$ by Proposition \ref{np}. Hence $I = 0$, which is a contradiction! Thus $\ord(I)\le \ell\ell(R)-1$.
\\
Now if $R$ has minimal multiplicity, then $e(R)=\mu(\m)-\dim R+1$. Since $\dim R=1$, so $1\leq \ord(I)\le \ell\ell(R)-1\le e(R)-\mu(\m)+1=1$. Thus we have equality throughout. 
\end{proof}

%Using \cite[Remark 1.5, Theorem 3.5(ii)]{kl} and Corollary \ref{5}, we get the following: 

Let $R$ be local and $M$ be an $R$-module. Note that, $\ord(\tr_R(M))=0$ if and only if $\tr_R(M)=R$ if and only if $R$ is a direct summand of $M$ (see \cite[Proposition 2.8.(iii)]{lindo}). 

The following result is a higher dimension version of Proposition \ref{her}. 

\begin{thm}\label{42} Let $(R,\m)$ be a local Cohen--Macaulay ring of depth $t$, with infinite residue field, and $M$ be an $R$-module such that $M^*\neq 0$, $R$ is not a direct summand of $M$, and $\Ext^{1\le i \le t-1}_R(M,R)=0$. Then $1\le \ord\left(\tr_R(M)\right)\le \ell\ell(R)-1\le e(R)-\mu(\m)+t$. In particular, equality holds throughout when $R$ has minimal multiplicity. 
\end{thm}       

\begin{proof} We first consider the Artinian case $t=0$. If $\tr_R( M)\subseteq \m^{\ell\ell(R)}=0$, then  $ M^*=0$, contradiction! Thus $1\leq \ord\left(\tr_R( M)\right)\le \ell\ell(R)-1\le e(R)-\mu(\m)$ (see \cite[Remark 1.5, Theorem 3.5(ii)]{kl}).

Now we consider $t\geq 1$.  First note that, $\ell\ell(R)-1\le e(R)-\mu(\m)+\dim R$ follows from \cite[Remark 1.5, Theorem 3.5(ii)]{kl}. Now we will prove $\ord\left(\tr_R(M)\right)\le \ell\ell(R)-1$ by contradiction. So, if possible let, $\ord\left(\tr_R(M)\right)\ge \ell\ell(R)$. This implies $\tr_R(M)\subseteq \m^{\ell\ell(R)}$. Now by the definition of $\ell\ell(R)$ we have, $\m^{\ell\ell(R)}\subseteq (\textbf{x})$ for some system of parameter $\textbf{x}=x_1,\cdots,x_t$. Thus $\tr_R(M)\subseteq (\textbf{x})$. Now let $I:=(\textbf{x})$. Since $\pd_R(I)=t-1$ and $\Ext^{1\le i \le t-1}_R(M,R)=0$, so Theorem \ref{hah}(1) implies that $M^*=0$, which is a contradiction! Hence $\ord\left(\tr_R(M)\right)\le \ell\ell(R)-1$.   
\\
Now if $R$ has minimal multiplicity, then $e(R)=\mu(\m)-\dim R+1$. So, $1\leq \ord\left(\tr_R(M)\right)\le \ell\ell(R)-1\le e(R)-\mu(\m)+\dim R=1$. Thus we have equality throughout.
\end{proof}

\begin{cor} Let $(R,\m)$ be a local Cohen--Macaulay ring of depth $t$, with infinite residue field, and $M$ be an $R$-module such that $M\neq 0$, $R$ is not a direct summand of $M$, and $\Ext^{1\le i \le t}_R(M,R)=0$. Then $1\le \ord\left(\tr_R(M)\right)\le \ell\ell(R)-1\le e(R)-\mu(\m)+\dim R$. In particular, equality holds throughout when $R$ has minimal multiplicity.
\end{cor}
\begin{proof} Since $\Ext^{1\le i \le t}_R(M,R)=0$ and $M\neq 0$, so $M^*\neq 0$ by \cite[Proposition 1.2.10]{bh}. Now we are done by Theorem \ref{42}. 
\end{proof}

Let $(R,\m,k)$ be a local Cohen--Macaulay ring of dimension $1$ with $k$ infinite. If $I$ is an $\m$-primary ideal of $R$, then $\c\subseteq \tr_R(I)$ (\cite[Lemma 3.2]{med}), hence $\ord(\tr_R(I))\le \ord(\c)$. We also notice that $\c=\tr_R(\overline R)$ is itself a trace ideal. Motivated by Proposition \ref{her}, we then raise the following question:  

\begin{ques}\label{hyp} Let $(R,\m)$ be a complete local hypersurface domain of dimension $1$. Then is $\ord(\c)=e(R)-1$? 
\end{ques}

For numerical semigroup rings, Question \ref{hyp} has an affirmative answer as follows:   

\begin{prop}\label{nuco} Let $a<b$ be relatively prime positive integers. Consider $R=k[[t^a,t^b]]$, where $k$ is a field. Then $\ord(\c)=a-1$. 
\end{prop}

\begin{proof} If $a=1$, then there is nothing to prove. Now assume $a>1$. By \cite[Example 12.2.1]{sh} we have, $\c=\left(t^i: i\geq (a-1)(b-1)\right)$. Let $i\geq (a-1)(b-1)$. As $t^i\in R$ for $i\geq (a-1)(b-1)$ so $i=ra+sb$ for $r,s\geq 0$. Now if $r+s\leq a-2$, then $(a-1)(b-1)\leq ra+sb\leq(r+s)b\leq (a-2)b$. So, $ab-a-b+1\leq ab-2b$, so $b+1\leq a$, contradiction!  Thus, if $i\geq (a-1)(b-1)$, then $i=ra+sb$, where $r,s\geq 0$ and $r+s\geq a-1$. Hence $t^i=(t^a)^r(t^b)^s\in \m^{r+s}\subseteq \m^{a-1}$. Hence $\c \subseteq \m^{a-1}$, i.e., $\ord(\c)\geq a-1$. Now $e(R)=a$ by \cite[Exercise 4.6.18]{bh}. So, the claim follows from Proposition \ref{her}.

%We will show that, $\ord(\c)=2=e(R)-1$. We know that $\c=(\{t^i:i\ge 2(b-1)\})$, so we will show that $(\{t^i:i\ge 2(b-1)\})\subseteq \m^2= (t^3,t^b)^2=(t^6,t^{3+b},t^{2b})$. So, it is enough to show that $\{t^i:i\ge 2(b-1)\}\subseteq (t^6,t^{3+b},t^{2b})$. Now since $t^3\in R$, so it is enough to show that $\{t^{2(b-1)},t^{2(b-1)+1},t^{2(b-1)+2}\}\subseteq (t^6,t^{3+b},t^{2b})$. But always $t^{2(b-1)+2}=t^{2b}\in (t^6,t^{3+b},t^{2b})$, so it is enough to show that $t^{2(b-1)},t^{2(b-1)+1}\in (t^6,t^{3+b},t^{2b})$.
% \\
% Since gcd$(b,3)=1$, so $b$ is either of the form $3k+1$ for some integer $k> 0$ or of the form $3k+2$ for some integer $k\ge 0$. First let, $b=3k+1$ for some integer $k> 0$. Then $2(b-1)=2(3k+1-1)=6k$ and $2(b-1)+1=6k+1=(3+(3k+1))+3(k-1)=(3+b)+3(k-1)$. So, $t^{2(b-1)},t^{2(b-1)+1}\in (t^6,t^{3+b},t^{2b})$, so $\c=(\{t^i:i\ge 2(b-1)\})\subseteq \m^2=(t^6,t^{3+b},t^{2b})$.
% \\
% Next let, $b=3k+2$ for some integer $k\ge 0$. Then $2(b-1)=2(3k+2-1)=6k+2=(3+(3k+2))+3(k-1)=(3+b)+3(k-1)$ and $2(b-1)+1=6k+3$. So, $t^{2(b-1)},t^{2(b-1)+1}\in (t^6,t^{3+b},t^{2b})$, so $\c=(\{t^i:i\ge 2(b-1)\})\subseteq \m^2=(t^6,t^{3+b},t^{2b})$.
\end{proof}

Since for one-dimensional local Cohen--Macaulay rings of minimal multiplicity we have $e(R)=\mu(\m)$, so the next boundary case of Proposition \ref{her} gives rise to the following question. 

\begin{ques}\label{qu2} Let $R=k[[H]]$ be a numerical semigroup ring with $e(R)-\mu(\m)=1$. Then is $\ord(\c)=2$?  
\end{ques}

In relation to Question \ref{qu2}, we mention, in passing, a straightforward lower bound for the order of the conductor of a numerical semigroup ring in terms of the largest generator of the minimal generating set and the Frobenius number $F(S)$ (see \cite[Chapter 1, Section 3]{gs}) of the numerical semigroup.  

\begin{prop} Let $H=\langle a_1,\cdots,a_n\rangle$ be a numerical semigroup and $k[[t^{a_1},\cdots, t^{a_n}]]$ be a corresponding numerical semigroup ring, where $1<a_1<\cdots<a_n$. Then $\ord(\c)\geq \left\lfloor\frac{F(S)}{a_n}\right\rfloor+1$. In particular, if $F(S)>a_n$, then $\ord(\c)\geq 2$.
\end{prop} 

\begin{proof} Set $b:=\left\lfloor\frac{F(S)}{a_n}\right\rfloor$. Then $F(S)+1>ba_n$. We know that, $\c=(t^i: i\geq F(S)+1)$. For every $i\geq F(S)+1$, $i=\sum_{j=1}^n c_ja_j\leq (\sum_{j=1}^n c_j)a_n$ for some non-negative integers $c_1,\cdots,c_n$. Thus $(\sum_{j=1}^n c_j)a_n\geq F(S)+1>ba_n$, and so $\sum_{j=1}^n c_j>b$, i.e., $\sum_{j=1}^n c_j\geq b+1$. Hence $t^i=\prod_{j=1}^n (t^{a_j})^{c_j}\in \m^{\sum_{j=1}^nc_j}\subseteq \m^{b+1}$. So, $\ord(\c)\geq b+1$, i.e., $\ord(\c)\geq \left\lfloor\frac{F(S)}{a_n}\right\rfloor+1$.    
\end{proof}

\section{order of ideals generated by entries of maps in minimal free resolutions}\label{minfreeord}

In this section, all modules are assumed to be finitely generated.

We conclude with providing some upper bounds for the $\m$-adic order of the ideal of entries of certain maps in a minimal free resolution of certain modules over a local Cohen--Macaulay ring $(R,\m)$, which are direct consequences of results from Sections \ref{trid} and \ref{ortr}. All modules are assumed to be finitely generated in this section. Here, $\syz_R(-)$ will stand for syzygy in a minimal free resolution.  

\begin{chunk}\label{ords} If $(F_i,\partial_i)_{i=0}^\infty$ is a minimal free resolution of a module $M$ over a local ring, then $\I_1(\partial_i)\subseteq \mathfrak m$, and consequently $\ord(\I_1(\partial_i))\geq 1$ for every $i\ge 1$.  
\end{chunk}

\begin{cor}\label{62} Let $(R,\m)$ be a local Cohen--Macaulay ring of depth $t\geq 1$, with infinite residue field, and $M$ be a non-free $R$-module such that $R$ is not a direct summand of $\syz_R M$, and $\Ext^{1\le i \le t}_R(M,R)=0$. If $(F_i,\partial_i)_{i=0}^\infty$ is a minimal free resolution of $M$, then $1\le \ord\left(\I_1(\partial_1)\right)\le \ell\ell(R)-1\le e(R)-\mu(\m)+t$. In particular, equality holds throughout when $R$ has minimal multiplicity.
\end{cor}

\begin{proof} As $M$ is not free, so $\syz_R M\ne 0$. Hence $(\syz_R M)^*\neq 0$. The hypothesis yields $\Ext^{1\le i \le t-1}_R(\syz_R M, R)=0$. Hence applying Theorem \ref{42} on $\syz_R M$, we get $1\le \ord\left(\tr_R(\syz_R M)\right)\le \ell\ell(R)-1\le e(R)-\mu(\m)+\dim R$. Since $t\geq 1$, so $\Ext^1_R(M,R)=0$. So, we get $\I_1(\partial_1)=\tr(\syz_R M)$ by Lemma \ref{43}(2).  
\end{proof}

Corollary \ref{62} does not cover the Artinian case, for which we have the following.

\begin{prop} Let $(R,\m)$ be an Artinian local ring, with infinite residue field, and $M$ be a non-free $R$-module such that $R$ is not a direct summand of $\syz_R M$, and $\Ext^{1}_R(M,R)=0$. If $(F_i,\partial_i)_{i=0}^\infty$ is a minimal free resolution of $M$, then $1\le \ord\left(\I_1(\partial_1)\right)\le \ell\ell(R)-1\le e(R)-\mu(\m)$. In particular, equality holds throughout when $R$ has minimal multiplicity. 
\end{prop}

\begin{proof} As $M$ is not free, so $\syz_R M\ne 0$.    If $\tr_R(\syz_R M)\subseteq \m^{\ell\ell(R)}=0$,  so $(\syz_R M)^*=0$, hence $\syz_R M=0$, contradiction! Thus $1\leq \ord\left(\tr_R(\syz_R M)\right)\le \ell\ell(R)-1\le e(R)-\mu(\m)$ (see \cite[Remark 1.5, Theorem 3.5(ii)]{kl}). Since $\Ext^1_R(M,R)=0$, so we get $\I_1(\partial_1)=\tr(\syz_R M)$ by Lemma \ref{43}(2). 
\end{proof}

% \if0
% \begin{cor}\label{622} Let $(R,\m)$ be a local Cohen--Macaulay ring of depth $t$, with infinite residue field, and $M$ be a non-free $R$-module such that $R$ is not a direct summand of $\syz_R M$, and $\Ext^{1\le i \le t+1}_R(M,R)=0$. If $(F_i,\partial_i)_{i=0}^\infty$ is a minimal free resolution of $M$, then $1\le \ord\left(\I_1(\partial_1)\right)\le \ell\ell(R)-1\le e(R)-\mu(\m)+\dim R$. In particular, equality holds throughout when $R$ has minimal multiplicity.
% \end{cor}

% \begin{proof} As $M$ is not free, so $\syz_R M\ne 0$. The hypothesis yields $\Ext^{1\le i \le t}_R(\syz_R M, R)=0$. Hence applying Theorem \ref{42} on $\syz_R M$, we get $1\le \ord\left(\tr_R(\syz_R M)\right)\le \ell\ell(R)-1\le e(R)-\mu(\m)+\dim R$. Since $\Ext^1_R(M,R)=0$, so we also get $\I_1(\partial_1)=\tr(\syz_R M)$ by Lemma \ref{43}(2).  
% \end{proof}
% \fi  

The technical assumption of $R$ not being a summand of $\syz_R M$ in Corollary \ref{62} is satisfied when, for instance, $M$ is maximal Cohen--Macaulay (\cite[Lemma 2.2]{del}).

\begin{chunk}\label{nchu} Let $I$ be a non-principal ideal of a local ring $R$  containing a non-zerodivisor. Let $(F_i,\partial_i)_{i=0}^\infty$ be a  free resolution of $I$. Then it follows from  Proposition \ref{trentry} that $\ord(\I_1(\partial_1))\le \ord(\tr_R(I))$.  
\end{chunk}

\begin{cor} Let $R$ be a local Cohen--Macaulay ring of dimension $1$ with infinite residue field. Let $I\subseteq \m$ be a non-principal ideal of $R$ such that $I=\tr_R(I)$ contains a non-zerodivisor. Let $(F_i,\partial_i)_{i=0}^\infty$ be a minimal free resolution of $I$. Then $1\le \ord\left(\I_1(\partial_1)\right)\le \ell\ell(R)-1\le e(R)-\mu(\m)+1$. In particular, equality holds throughout when $R$ has minimal multiplicity. 
\end{cor}

\begin{proof} Immediate from \ref{ords}, \ref{nchu} and Proposition \ref{her}. 
\end{proof} 

%\begin{rem} If $I$ is an ideal of $R$ such that $R$ is a direct summand of $I$, then $I\cong R$.  \end{rem}

\section{Some additional results on obstruction of containment of trace ideals}\label{sec7} 
In this section, all modules are assumed to be finitely generated. We will prove various results on obstruction of containment of ideals defining the singular locus of the ring in certain ideals of finite projective, or injective or $G$-dimension. Unlike the results in Section \ref{secmain},  we believe that none of the results in this section can be deduced from Theorem \ref{4}.

Before proceeding, we remark that $\sing(R)$ (respectively $\ng(R)$) will denote the collection of all $\p\in \spec(R)$ such that $R_{\p}$ is not regular (respectively, not Gorenstein). $\reg(R)$ will stand for $\spec(R)\setminus \sing(R)$.

\begin{lem}\label{new} Let $(R,\m)$ be local, and $M$ be a finitely generated $R$-module such that $\m M\neq 0$ and $\id_R \m M <\infty$. Then $R$ is regular.
\end{lem}

\begin{proof}  Follows from \cite[Example 3.4, Corollary 3.19]{burch}.   
\end{proof}

\begin{prop}\label{sin} Let $R$ be a local ring. Let $0\neq J$ be an ideal of $R$ such that $\V(J)\subseteq \sing(R)$. Let $I, L$ be ideals of $R$, where $I$ is a proper radical ideal. Let $n\ge 1$ be an integer such that $J\subseteq I^{(n)}L$.  Then  $\pd_R I^{(n)}L=\infty=\id_R I^{(n)}L$. 
\end{prop} 

\begin{proof} As $I$ is a proper radical ideal, hence $\Min(R/I)\neq \emptyset$ and $I=\cap_{\p \in \Min(R/I)}\p$. Thus $I^{(n)}=\cap_{\p \in \Min(R/I)}\p^{(n)}$, and $I^{(n)}R_{\p}=\p^nR_{\p}=\p^{(n)}R_{\p}$ for all $\p\in \Min(R/I)$.   

 By hypothesis, we have $J\subseteq I^{(n)}L\subseteq I^{(n)}\subseteq \p$ for all $\p\in \Min(R/I)$. Hence $\p\in \V(J)\subseteq \sing(R)$ for all $\p\in \Min(R/I)$. As  $J\neq 0$, hence $0\neq I^{(n)}L$. Now if possible let, $\pd_R I^{(n)}L<\infty$, i.e., $\pd_R R/(I^{(n)}L)<\infty$. By \cite[Proposition 6.2]{Ab} we get that, $I^{(n)}L$ contains a non-zero-divisor. Since localization commutes with product, hence $\p^nLR_{\p}=I^{(n)}LR_{\p}\neq 0$  for all $\p\in \Min(R/I)$. Now $\pd_R I^{(n)}L<\infty$ also implies $\pd_{R_{\p}} I^{(n)}LR_{\p}=\pd_{R_{\p}} \p^nLR_{\p}<\infty$ for all $\p\in \Min(R/I)$, hence $R_{\p}$ is regular by \cite[Theorem 1.1]{lv}. This finally contradicts $\p\in \sing(R)$, hence forcing $\pd_R I^{(n)}L=\infty$.

 The proof for $\id_R I^{(n)}L=\infty$ is similar by using Lemma \ref{38} in place of \cite[Proposition 6.2]{Ab}, and Lemma \ref{new} in place of \cite[Theorem 1.1]{lv}.
\end{proof}

For an $R$-module $M$, we put $\ipd_R(M):=\{\p \in \spec(R)| \pd_{R_{\p}} M_{\p}=\infty\}$ and  $\iid_R(M):=\{\p \in \spec(R)| \id_{R_{\p}} M_{\p}=\infty\}$. 

\begin{prop}\label{intg} Let $R$ be a ring. Let $ J$ be an ideal of $R$ such that $\V(J)\subseteq \sing(R)$. Let  $I$ be an integrally closed proper ideal of $R$ such that $J\subseteq I$.  Then $\Min(R/I)\subseteq \ipd_R(I) \cap \iid_R(I)$. In particular,  $\pd_R I=\infty=\id_R(I)$.  
\end{prop}  

\begin{proof} As $I$ is a proper ideal, hence $\Min(R/I)\neq \emptyset$. By hypothesis, we have $J\subseteq I\subseteq \p$ for all $\p\in \Min(R/I)$. Hence, $\p\in \V(J)\subseteq \sing(R)$ for all $\p\in \Min(R/I)$.  Note that, $I_{\p}$ is $\p R_{\p}$-primary and integrally closed (\cite[Proposition 1.1.4]{sh}) for every $\p\in \Min(R/I)$. Now $\pd_{R_{\p}} I_{\p}=\infty=\id_{R_{\p}}I_{\p}$ follows from \cite[Corollary 6.12]{rigid}. 
\end{proof}

\begin{lem}\label{rfd} Let $M$ be an $R$-module such that $\reg(R)\subseteq \supp(M)$ and $\depth_{R_{\p}} M_{\p} \ge \depth R_{\p} $ for all $\p \in \reg(R)$. Then $\V(\tr_R M)\subseteq \sing(R)$.
\end{lem}

\begin{proof} If $\p \notin \sing(R)$, then $R_{\p}$ is regular and $M_{\p}\neq 0$. Then $\depth_{R_{\p}} M_{\p} \ge \depth R_{\p} $  implies $0\neq M_{\p}$ is $R_{\p}$-free. Hence $R_{\p}=\tr_{R_{\p}} (M_{\p})=(\tr_R M)_{\p}$ (\cite[Proposition 2.8(ii),(viii)]{lindo}). Thus $\tr_R M \nsubseteq \p$ , i.e., $\p \notin \V(\tr_R M)$.   
\end{proof}

Using Proposition \ref{sin}, Proposition \ref{intg} and Lemma \ref{rfd}, we obtain another result on obstruction of containment of trace ideals in certain ideals of finite homological dimension.

\begin{thm}\label{75} Let $M$ be an $R$-module of full support and $\depth_{R_{\p}} M_{\p} \ge \depth R_{\p} $ for all $\p \in \reg(R)$. Let $I, L$ be ideals of $R$, where $I$ is a proper  ideal. Then the following hold true:

\begin{enumerate}[\rm(1)]

    \item If $R$ is local, $I$ is a radical ideal, $n\ge 1$ is an integer such that $\tr_R (M)\subseteq I^{(n)}L$, then $\pd_R I^{(n)}L=\infty=\id_R I^{(n)}L$.

    \item If $I$ is integrally closed and $\tr_R (M)\subseteq I$, then $\Min(R/I)\subseteq \ipd_R(I)\cap \iid_R(I)$. In particular, $\pd_R I=\infty=\id_R(I)$.    
\end{enumerate}
    
\end{thm}

\begin{proof} 

(1) As $M$ has full support, hence $\Ass(M^*)=\supp(M)\cap \Ass(R)\neq \emptyset$. So, $M^*\neq 0$, hence $\tr_R(M)\neq 0$. Now we are done by Lemma \ref{rfd} and Proposition \ref{sin}.   

(2)  Immediate from Proposition \ref{intg} and Lemma \ref{rfd}. 
\end{proof}

Next, we aim to prove certain G-dimension versions of Proposition \ref{sin} and \ref{intg}. We refer the reader to \cite{ch} for basics of G-dimension. 

For an $R$-module $M$, we set $\igd_R(M):=\{\p \in \spec(R)| \gdim_{R_{\p}} M_{\p}=\infty\}$.

\begin{prop}\label{intgor} Let $R$ be a ring. Let $ J$ be an ideal of $R$ such that $\V(J)\subseteq \ng(R)$. Let  $I$ be an integrally closed proper ideal of $R$ such that $J\subseteq I$.  Then $\Min(R/I)\subseteq \igd_R(I)$. In particular, $\gdim_R I=\infty$.   
\end{prop}  

\begin{proof} As $I$ is a proper ideal, hence $\Min(R/I)\neq \emptyset$. By hypothesis, we have $J\subseteq I\subseteq \p$ for all $\p\in \Min(R/I)$. Hence $\p\in \V(J)\subseteq \ng(R)$ for all $\p\in \Min(R/I)$.  Note that, $I_{\p}$ is $\p R_{\p}$-primary and integrally closed (\cite[Proposition 1.1.4]{sh}) for every $\p\in \Min(R/I)$. Now $\gdim_{R_{\p}} I_{\p}=\infty$ follows from \cite[Theorem 1.1]{cw}. 
\end{proof}

\begin{lem}\label{gd} Let $(R,\m)$ be local Cohen--Macaulay, and $M$ be a finitely generated $R$-module such that $\m M\neq 0$ and $\gdim_R \m M <\infty$. Then $R$ is Gorenstein. 
\end{lem}

\begin{proof} We may pass to completion and assume $R$ is complete, with canonical module $\omega_R$. Then $\Tor^R_{\gg 0}(\m M, \omega_R)=0$ by \cite[Theorem 5.2.14]{ch}. Now we are done by \cite[Example 3.4, Proposition 3.16]{burch}. 
\end{proof}

\begin{prop}\label{ngr} Let $R$ be a local Cohen--Macaulay ring. Let $ J$ be an ideal of $R$ containing a non-zero-divisor such that $\V(J)\subseteq \ng(R)$. Let $I, L$ be ideals of $R$, where $I$ is a proper radical ideal. Let $n\ge 1$ be an integer such that $J\subseteq I^{(n)}L$.  Then  $\gdim_R I^{(n)}L=\infty$. 
\end{prop} 

\begin{proof} As $I$ is a proper radical ideal, hence $\Min(R/I)\neq \emptyset$ and $I=\cap_{\p \in \Min(R/I)}\p$. Thus $I^{(n)}=\cap_{\p \in \Min(R/I)}\p^{(n)}$, and $I^{(n)}R_{\p}=\p^nR_{\p}=\p^{(n)}R_{\p}$ for all $\p\in \Min(R/I)$.   

 By hypothesis, we have $J\subseteq I^{(n)}L\subseteq I^{(n)}\subseteq \p$ for all $\p\in \Min(R/I)$. Hence $\p\in \V(J)\subseteq \ng(R)$ for all $\p\in \Min(R/I)$. As  $J\neq 0$, hence $0\neq I^{(n)}L$. Now if possible let, $\gdim_R I^{(n)}L<\infty$. As $J$ contains a non-zero-divisor, so does $I^{(n)}L$. Since localization commutes with product, hence,  $\p^nLR_{\p}=I^{(n)}LR_{\p}\neq 0$  for all $\p\in \Min(R/I)$. Now $\gdim_R I^{(n)}L<\infty$ also implies $\gdim_{R_{\p}} I^{(n)}LR_{\p}=\gdim_{R_{\p}} \p^nLR_{\p}<\infty$ for $\p\in \Min(R/I)$. Hence $R_{\p}$ is Gorenstein by Lemma \ref{gd}. This finally contradicts $\p\in \ng(R)$, hence forcing $\gdim_R I^{(n)}L=\infty$.

 % The proof for $\id_R I^{(n)}L=\infty$ is similar by using Lemma \ref{38} in place of \cite[Proposition 6.2]{Ab}, and Lemma \ref{new} in place of \cite[Theorem 1.1]{lv}.

\end{proof}

\begin{thm} Let $R$ be a local Cohen--Macaulay ring admitting a canonical ideal $\omega_R$. Let $I, L$ be ideals of $R$, where $I$ is a proper  ideal. Then the following hold true:

\begin{enumerate}[\rm(1)]

    \item If  $I$ is a radical ideal, $n\ge 1$ is an integer such that $\tr_R (\omega_R)\subseteq I^{(n)}L$, then $\gdim_R I^{(n)}L=\infty$.

    \item If $I$ is integrally closed and $\tr_R (\omega_R)\subseteq I$, then $\Min(R/I)\subseteq \igd_R(I)$. In particular, $\gdim_R I=\infty$.     
\end{enumerate}
    
\end{thm}

\begin{proof} We first recall that $\V(\tr(\omega_R))=\ng(R)$ (\cite[Lemma 2.1]{he}). (1) now follows from this and Proposition \ref{ngr}. (2) follows from this and Proposition \ref{intgor}.
    
\end{proof}

\end{document}